\documentclass[12pt,reqno]{amsart}

\usepackage{mathabx}
\usepackage{comment}
\usepackage{multicol}
\usepackage[inline]{asymptote}
\usepackage[hidelinks]{hyperref}
\usepackage{csquotes}
\usepackage{url}
\usepackage{amsmath,amssymb}
\usepackage{graphicx}
\usepackage{centernot}
\usepackage{mathtools}
\usepackage{stmaryrd}
\usepackage{tabularx}
\usepackage{tikz}
\usepackage[shortlabels]{enumitem}
\usepackage[letterpaper, width=165mm, top=25mm,bottom=25mm,bindingoffset=6mm]{geometry}

\theoremstyle{plain}
\newtheorem{theorem}{Theorem}
\numberwithin{theorem}{section}
\newtheorem{lemma}[theorem]{Lemma}

\newtheorem{proposition}[theorem]{Proposition}

\newtheorem{problem}[theorem]{Problem}
\theoremstyle{definition}
\newtheorem{definition}[theorem]{Definition}

\newtheorem{notation}[theorem]{Notation}
\newtheorem{remark}[theorem]{Remark}
\newtheorem{example}[theorem]{Example}

\newtheorem{case}{Case}
\newtheorem{subcase}{Subcase}
\numberwithin{subcase}{case}
\newtheorem{subsubcase}{Subsubcase}
\numberwithin{subsubcase}{subcase}

\allowdisplaybreaks

\newcommand{\IN}{\mathbb{N}}
\newcommand{\IZ}{\mathbb{Z}}
\newcommand{\IQ}{\mathbb{Q}}
\newcommand{\IF}{\mathbb{F}}
\newcommand{\CT}{\operatorname{CT}}
\newcommand{\Comp}{\operatorname{Comp}}
\newcommand{\BU}{\operatorname{BU}}
\newcommand{\ite}{\operatorname{it}}
\newcommand{\Sym}{\operatorname{Sym}}
\newcommand{\id}{\operatorname{id}}
\newcommand{\FSR}{\operatorname{FSR}}
\newcommand{\lcm}{\operatorname{lcm}}
\newcommand{\AGL}{\operatorname{AGL}}
\newcommand{\GL}{\operatorname{GL}}
\newcommand{\pow}{\operatorname{pow}}
\newcommand{\ord}{\operatorname{ord}}

\newcommand{\fcp}{\operatorname{fcp}}
\newcommand{\im}{\operatorname{im}}
\newcommand{\wt}{\operatorname{wt}}

\makeatletter

\begin{document}

\title[Wreath products and cascaded FSRs]{Wreath products and cascaded FSRs}
\author[Alexander Bors, Farzad Maghsoudi, and Qiang Wang]{Alexander Bors, Farzad Maghsoudi, and Qiang Wang}
\address{School of Mathematics and Statistics, Carleton University, Ontario, Ottawa, K1S 5B6, Canada}
\email{alexanderbors@cunet.carleton.ca,
farzadmaghsoudi@cmail.carleton.ca,\newline wang@math.carleton.ca}
\subjclass[]{}
\begin{abstract}
 We show that the transition function of the cascaded connection of two FSRs can be viewed as a wreath product element. This allows us to study periods of cascaded connections with algebraic methods, obtaining both a general, nontrivial upper bound on the maximum period of a cascaded connection and a complete, explicit understanding of the periods in the important case of the cascaded connection of an $n$-dimensional De Bruijn sequence into an $m$-dimensional linear FSR.
\end{abstract}

{\mathversion{bold} \maketitle}

\numberwithin{theorem}{section}
\section{Introduction}
Linear feedback shift registers~(LFSRs) have a wide range of applications in coding theory, modern communication systems, and cryptography.~There has been substantial use of LFSRs as building blocks in stream ciphers because of their very good statistical properties, efficient implementations, and well-studied algebraic structures.~In contrast, stream ciphers based on LFSRs are vulnerable to correlation attacks~\cite{Aumasson} and algebraic attacks~\cite{Cannière}.~Consequently, nonlinear feedback shift registers (NFSRs) have attracted increasing attention for their nonlinear update functions.~Recently proposed stream ciphers, such as Trivium~\cite{Christophe} and Grain~\cite{Nicolas}, use NFSRs as building blocks.~Over the past 50 years, NFSRs have been examined. However, some fundamental questions remain unanswered.~In particular, there is no efficient way to determine the periods of sequences generated by an arbitrary NFSR.~The most important kind of sequence generated by an NFSR, which achieves the maximum period, of $2^n$ in the case of an $n$-stage NFSR, is a~\textit{de Bruijn} sequence. LFSRs can also achieve sequences of large periods, namely up to $2^n-1$ when having $n$ stages; the corresponding output sequences are known as $m$-sequences, and the associated transition functions are special cases of Singer cycles. However, each of these two classes of feedback shift registers (FSRs) has drawbacks from an application point of view:
\begin{enumerate}
\item De Bruijn sequences with desirable properties are costly to generate for large $n$.
\item Linearity is a property that should be avoided for many applications.
\end{enumerate}
It is thus advantageous to have a method of combining smaller FSRs, e.g.~a De Bruijn cycle and an LFSR, into larger contraptions which may achieve long cycles in higher dimensions more efficiently than De Bruijn cycles while performing better than LFSRs with respect to certain undesirable properties. This is where cascaded connections of FSRs come into play.

The cascaded connection of two FSRs was first introduced in~\cite{Green}. Specifically, the cascaded connection of $\FSR(f) $ into $\FSR(g)$ produces the same family of sequences as the FSR with characteristic function $f\ast g $~\cite{Green}. This was a motivation for Mykkeltveit, Siu and Tong to study some properties of the cycle structure of the cascaded connection of $\FSR(f)$ into $\FSR(g)$~\cite{Mykkeltveit}.~A Grain-like structure is a cascaded connection of a primitive LFSR into an NFSR.~In 2011, Hu and Gong demonstrated that the periods of the sequences generated by an NFSR in a Grain-like structure are multiples of the periods of the sequences generated by its LFSR~\cite{Hu}.~They also proposed an open problem whether the sequences generated by the NFSR in a Grain-like structure can achieve the minimum period, i.e., the period of the LFSR.~In terms of security, it is clearly undesirable for sequences generated by NFSRs in Grain-like structures to achieve the minimum period.

Recently, in 2019, Yang, Zeng and Xu in~\cite{Yang} proved the existence of a Grain-like structure achieving the minimum period by constructing a class of them for the case where the LFSR and NFSR have the same number of stages.~Inspired by their work, Wang, Zheng, Zhao and Feng in~\cite{Wang} could improve their result and also prove the existence of such Grain-like structures for the case where the number of stages of the NFSR is larger than the number of stages of the LFSR.~They also proved that there are two necessary conditions for Grain-like structures to generate maximal possible period sequences.

In 2011, Cheng, Qi and Li proposed a new mathematical tool for calculating matrices called semi-tensor products (STP)~\cite{Analysis}.~The STP method has been widely used to study Boolean networks -- see the survey papers \cite{Li, Lu} for more information. This method is also used in the study of cascaded connections of FSRs.~Particularly, in ~\cite{Liu}, using the STP of matrices, Grain-like cascaded FSRs are converted into an equivalent linear equation by declaring them as two Boolean networks.

Cascading an NFSR into an LFSR to generate long sequences has been studied as well.~In 2020, Chang, Gong and Wang using a linear algebraic approach obtained a description of the cycle structure of a cascaded connection of an arbitrary NFSR generating a de Bruijn sequence into an LFSR~\cite{Chang}.~In particular, they showed that the initial state of each cycle can be determined by solving a system of linear equations.~Their method also works for the cascaded connection of any NFSR into an LFSR.

The wreath product is a special combination of two permutation groups based on the semidirect product in group theory.~It is a fundamental concept in permutation group theory~\cite[Secion 2.6]{Dixon}. There are applications of the wreath product in describing cyclotomic permutations, computing the cycle types of certain permutations, etc. Wan and Lidl in \cite{Rudolf} used wreath products to study cyclotomic permutations. Based on this earlier work, Bors and Wang expanded those ideas in~\cite{Alexander}.~Using the imprimitive wreath product, they were also able to show that some specific functions form a permutation group on $K$ (a finite field) and characterize which of them are complete mappings of $K$~\cite{Bors}.

In this paper, we will use wreath products to study periods of cascaded connections of FSRs. This approach is more conceptual than the one of Mykkeltveit, Siu and Tong \cite[Section 2]{Mykkeltveit}, and for the special case of the cascaded connection of a De Bruijn cycle into an LFSR, it will yield results that are more explicit/stronger than the ones of Chang, Gong and Wang \cite{Chang}.

In~Section~\ref{ch:Preliminaries}, we explain how the transition function of a cascaded connection of two FSRs can be viewed as a wreath product element, and we keep this perspective throughout the rest of the paper. This is followed by Section \ref{ch:Auxiliaries}, in which we collect some auxiliary results, of algebraic nature, that are needed for our main results. Those main results are then formulated and proved in Section~\ref{results}. Specifically, in Theorem~\ref{thm 3.2}, we give a nontrivial upper bound on the maximum period that can be achieved by the cascaded connections of two FSRs. Moreover, we give a comprehensive analysis, via algebraic methods, of the possible cycle structures of cascaded connections of a De Bruijn cycle into a linear FSR, leading to Theorem \ref{longTheo}, an explicit description of those cycle structures. Finally, in Section \ref{examples}, we go through some computational examples to illustrate our method.
\section{Preliminaries}\label{ch:Preliminaries}

The purpose of this section is to introduce terminology and definitions needed throughout the paper.

\subsection{Feedback shift registers and cascaded connections}\label{subsec2P1}

An $n$-stage feedback shift register (FSR) with characteristic function
\begin{align*}
    f(x_0, x_1,\ldots, x_n) = f_1(x_0, x_1,\ldots, x_{n-1}) \oplus x_n 
\end{align*}
consists of $n$ binary storage devices called stages. Each stage associates
with a state variable $x_i \in \{0, 1\} $ which represents the current value of the stage. Moreover, the stages are arranged linearly, say from left to right, and are connected through wires with the neighboring stages as well as with a circuit representing the Boolean function $f_1$, called the \emph{update function} of the FSR, that feeds its output back to the rightmost stage. Finally, the leftmost stage is connected outward via wires; it represents the output bit in each iteration of the FSR. The following picture illustrates the situation.
\begin{figure}[!htb]\label{simplefsr}
    \includegraphics[width=120mm,height=28mm]{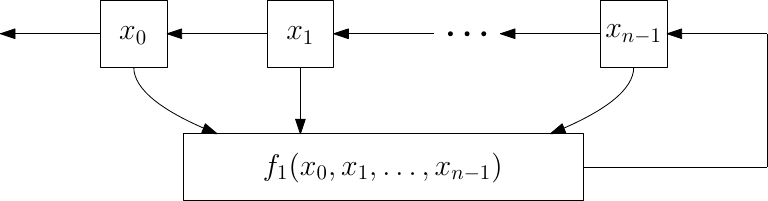}
    \caption{Basic set-up of an FSR.}
\end{figure}

In the beginning of the process of creating an output bit sequence with an FSR, its $n$ stages are initialized to certain bit values. Then, in each iteration of the process, the leftmost bit $x_0$ is output, all stages except the rightmost are overwritten with the stored bit value from their right neighbor, and the rightmost stage is overwritten with the bit $f_1(x_0,x_1,\ldots,x_{n-1})$. One may think of each stage of this process as encoded by the vector $(x_0,x_1,\ldots,x_{n-1})\in\IF_2^n$ of bits stored in the $n$ stages, and then the passing from one process iteration to the next is represented by an application of the so-called \emph{(state) transition function}
\[
\tilde{f}:\IF_2^n\rightarrow\IF_2^n, (x_0,x_1,\ldots,x_{n-1})\mapsto (x_1,x_2,\ldots,f_1(x_0,\ldots,x_{n-1})).
\]
Throughout this paper, we will use
\begin{itemize}
\item lower-case Latin letters such as $f$ or $g$ for the characteristic functions of FSRs;
\item those same letters with an added subscript $1$, as in $f_1$ or $g_1$, for the associated update function (see also below for the meaning of the corresponding notation with subscript $0$ instead of subscript $1$); and
\item those same letters with an added tilde, such as $\tilde{f}$ or $\tilde{g}$, to denote the associated transition function.
\end{itemize}
We note that an $n$-stage FSR with characteristic function $f$, which we also denote by $\FSR(f)$ as in the Introduction, is \emph{periodic}, i.e., always returns to any chosen initial state after a suitable number of iterations, if and only if its transition function $\tilde{f}$ is a permutation of $\IF_2^n$ if and only if the (Boolean) update function $f_1$ can be written in the form
\[
f_1(x_0,x_1,\ldots,x_{n-1})=x_0 \oplus f_0(x_1,x_2,\ldots,x_{n-1})
\]
for a suitable $(n-1)$-variate Boolean function $f_0$.

Next, we recall the definition of a cascaded connection of (two) FSRs.

\begin{definition}\label{def 2.1}{\cite[Definition 2]{Wang}}
Let $\FSR(f)$ be an $n$-stage FSR, and let $\FSR(g)$ be an $m$-stage FSR. The \emph{cascaded connection of $\FSR(f)$ into $\FSR(g)$}, denoted by $\FSR(f;g)$, is shown by Figure~\ref{Cascaded}.
\begin{figure}[!htb]\label{cascaded}
  \includegraphics[width=160mm,height=17mm]{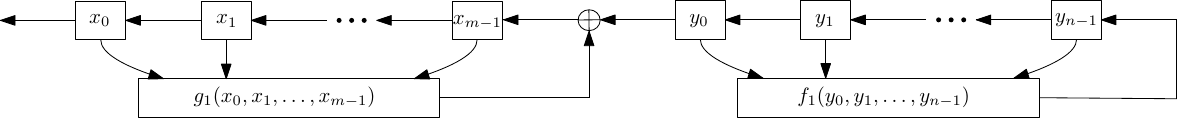}
  \caption{Cascaded Connection}\label{Cascaded}
\end{figure}
\end{definition}

Again, the leftmost bit in each iteration is the output bit, and we have a natural notion of transition function for $\FSR(f;g)$, namely the function
\begin{align*}
&\tilde{f}\ast\tilde{g}:\IF_2^{m+n}\rightarrow\IF_2^{m+n}, \\
&(x_0,\ldots,x_{m-1},y_0,\ldots,y_{n-1})\mapsto \\
&(x_1,\ldots,x_{m-1},g_1(x_0,\ldots,x_{m-1})+y_0,y_1,\ldots,y_{n-1},f_1(y_0,\ldots,y_{n-1})).
\end{align*}
The chosen notation $\tilde{f}\ast\tilde{g}$ makes sense because $f_1$ and $g_1$ can be derived from their associated transition functions $\tilde{f}$ and $\tilde{g}$, whence the transition function of $\FSR(f;g)$ can be viewed as an object depending on $\tilde{f}$ and $\tilde{g}$.

It was shown by Green and Dimond in ~\cite{Green} that the cascaded connection $\FSR(f;g)$ is equivalent (in the sense defined below) to the FSR whose characteristic function is the so-called $\ast$-product of $f$ into $g$ (to be distinguished from the $\ast$-product of transition functions introduced in the previous paragraph), which is the function in $n+m$ variables defined by the formula
\begin{align*}
  f\ast g = f(g(x_0, x_1,\ldots, x_m), g(x_1, x_2,\ldots, x_{m+1}),\ldots, g(x_n, x_{n+1},\ldots, x_{m+n})).
\end{align*}
Here, equivalence is to be understood as the existence of a bijective function $b:\mathbb{F}_2^{n+m}\rightarrow\mathbb{F}_2^{n+m}$ such that for each $\vec{x}\in\mathbb{F}_2^{n+m}$, the bit sequence produced by $\FSR(f\ast g)$ when initialized with $\vec{x}$ has the same period as the bit sequence produced by $\FSR(f;g)$ when initialized with $b(\vec{x})$.

In the next subsection, we provide a description of algebraic nature of the transition function $\tilde{f}\ast\tilde{g}$ of a cascaded connection $\FSR(f;g)$, which will be very useful in the sequel.

\subsection{Transition functions of cascaded connections as wreath product elements}The goal of this subsection is to explain how to view transition functions of cascaded connections of FSRs as wreath product elements, and how this enables an algebraic approach to study the periods of output sequences of cascaded connections. We start with the following elementary, but important observation, which reduces periods of output sequences to cycle lengths of transition functions:

\begin{lemma}\label{lemma 2.12}
Let $\FSR(f)$ be an $n$-stage FSR, and let $\FSR(g)$ be an $m$-stage FSR. Moreover, let $\Vec{v}=(x_0,\ldots,x_{m-1},y_0,\ldots,y_{n-1})^T\in\IF_2^{m+n}$. The following hold:
\begin{enumerate}
\item The output sequence produced by the cascaded connection $\FSR(f;g)$ when initialized with $\Vec{v}$ is periodic if and only if $\Vec{v}$ is a periodic point of the transition function $\tilde{f}\ast\tilde{g}$ (i.e., one eventually returns to $\Vec{v}$ after sufficiently many iterations of $\tilde{f}\ast\tilde{g}$).
\item Assume that $\Vec{v}$ is a periodic point of $\tilde{f}\ast\tilde{g}$. Then the cycle length of $\Vec{v}$ under $\tilde{f}\ast\tilde{g}$ equals the (least) period of the output sequence produced by $\FSR(f;g)$ when initialized with $\Vec{v}$.
\end{enumerate}
 \end{lemma}

\begin{proof}
For $ t \in \mathbb{N}_0 $, let us set $ \Vec{v}_t := (\tilde{f}\ast\tilde{g})^{t}(\Vec{v}) $. Also, let $ b_t $ denote the $ t $-th bit in the output sequence. By definition, $ b_t $ is the first entry of $ \Vec{v}_t $, and the following two statements follow from this:
\begin{itemize}
    \item If $\Vec{v}$ is a periodic point of $\tilde{f}\ast\tilde{g}$, then the associated output sequence $(b_t)_{t\geq0}$ is periodic.
    \item Under the assumption of statement (2), the period length of $ (b_t)_{t\geq 0} $ divides that of $ (\Vec{v}_t)_{t\geq 0} $, which is the cycle length of $ \Vec{v} $ under $\tilde{f}\ast\tilde{g}$.
\end{itemize}
On the other hand, for any $ t \in \mathbb{N}_0 $, consider the length $ m+n $ segment $ (b_t,b_{t+1},\ldots,b_{t+m+n-1}) $ of $ (b_t)_{t\geq 0} $. We prove that $ \Vec{v}_t $ can be reconstructed from that segment through applying a certain injective function. Indeed, $ (b_t,b_{t+1},\ldots,b_{t+m-1}) $ is equal to the length $ m $ initial segment of $ \Vec{v}_t = (v^{(t)}_0,v^{(t)}_1,\ldots,v^{(t)}_{m+n-1}) $, so the first $ m $ bits $v^{(t)}_0,v^{(t)}_1,\ldots,v^{(t)}_{m-1}$ in $ \Vec{v}_t $ can be directly read off. Moreover, for each $ k = 0,1,\ldots,n-1 $, it is easy to infer from the definitions of the involved concepts that
\[
b_{t+m+k} = g_1(b_{t+k},b_{t+k+1},\ldots,b_{t+m+k-1})+ v^{(t)}_{m+k},
\]
so that
\[
v^{(t)}_{m+k} = b_{t+m+k}+ g_1(b_{t+k},b_{t+k+1},\ldots,b_{t+m+k-1})
\]
can also be reconstructed. Therefore, there is an injective function $ \iota: \IF_2^{m+n} \rightarrow \IF_2^{m+n} $ such that $ \iota(b_t,b_{t+1},\ldots,b_{t+m+n-1}) = \Vec{v}_t $ for all $ t \geq 0 $. This implies the following two statements and concludes the proof:
\begin{itemize}
    \item If the output sequence $(b_t)_{t\geq0}$ is periodic, then $\Vec{v}$ is a periodic point of $\tilde{f}\ast\tilde{g}$.
    \item Under the assumption of statement (2), the period of $(\Vec{v}_t)_{t\geq 0} $, which is the cycle length of $ \Vec{v} $ under $\tilde{f}\ast\tilde{g}$, divides the period of the output sequence $(b_t)_{t\geq 0}$.\qedhere
\end{itemize} 
\end{proof}

Now we turn to explaining our algebraic approach for studying cycle lengths of transition functions of cascaded connections. First, we need to explain what a wreath product is.

Let $X$ and $\Omega$ be sets. We view the Cartesian product $X\times\Omega$ as a disjoint union of $|\Omega|$ copies of $X$, indexed by the elements of $\Omega$.

For example, if $\Omega=\{\omega_1, \omega_2, \omega_3\}$, then 
\begin{align*}
    X\times\Omega &= (X\times\{\omega_1\})~~\dot{\cup}~~(X\times\{\omega_2\})~~\dot{\cup}~~(X\times\{\omega_3\}) \\& = X_{\omega_1}~\dot{\cup}~X_{\omega_2}~\dot{\cup}~X_{\omega_3}.
\end{align*}
In the following definition and beyond, we use the exponent notation for function values, writing $x^f$ instead of $f(x)$. This is a common use of notation in group theory when $f$ is an element from a group that acts on a set of which $x$ is an element. We may still sometimes write $f(x)$ where this improves readability. We also use the notations $\textnormal{Sym}(X)$ to denote the symmetric group over the set $X$ (the group of all permutations of $X$ under function composition).

\begin{definition}
Let $ G \leq \textnormal{Sym}(X) $ and $ P \leq \textnormal{Sym}(\Omega)  $ be permutation groups. The \textit{(imprimitive) permutational wreath product} of $ G $ and $ P $, written $ G \wr P $, is the subgroup of $ \textnormal{Sym}(X \times \Omega) $ consisting of all permutations that preserve the above partition of $ X\times\Omega $, permuting the $ \Omega$-indexed blocks according to some element of $ P $. That is, each element of $ G \wr P $ is the composition of a permutation $ \sigma \in P \leq \textnormal{Sym}(\Omega) $ that permutes the blocks according to the rule $ (x, \omega) \mapsto (x, \omega^{\sigma}) $, and a tuple $ (f_{\omega})_{\omega \in \Omega} $ of permutations in $ G $ that permutes each block among itself, according to the rule $ (x, \omega) \mapsto (x^{f_{\omega}}, \omega) $.~The representation of an element of $ G \wr P $ in the form $ \sigma(f_{\omega})_{\omega \in \Omega} $ is unique.
\end{definition}

  In fact, $ \{(x, \omega) \mapsto (x, \omega^{\sigma}): \sigma \in P \} $ and $ \{(x, \omega) \mapsto (x^{f_{\omega}}, \omega) : (f_{\omega})_{\omega \in \Omega} \in G^{\Omega} \} $ are subgroups of $ G \wr P $, isomorphic to $ P $ and (the direct power) $ G^{\Omega} $, respectively, and $ G \wr P = P \ltimes G^{\Omega} $ (a semidirect product of $P$ and $G^{\Omega}$), where the conjugation action of $ P $ on the normal subgroup $ G^{\Omega} $ is by permuting coordinates, i.e.,
\begin{align*}
    (f_{\omega})_{\omega \in \Omega}^{\sigma} = \sigma^{-1}(f_{\omega})_{\omega \in \Omega}~\sigma = (f_{\sigma^{-1}(\omega)})_{\omega \in \Omega}.
\end{align*}
For example, if $ \Omega = \{\omega_1, \omega_2, \omega_3\} $ and $\sigma$ is the $3$-cycle $ (\omega_1, \omega_2, \omega_3) $, then 
\begin{align*}
    (f_{\omega_1}, f_{\omega_2}, f_{\omega_3})^{\sigma} = (f_{\sigma^{-1}(\omega_1)}, f_{\sigma^{-1}(\omega_2)}, f_{\sigma^{-1}(\omega_3)}) = (f_{\omega_3}, f_{\omega_1}, f_{\omega_2}).
\end{align*}

We note that the imprimitive wreath product of permutation groups has a natural generalization to transformation semigroups, which will also be needed in the sequel: if $G$ is a subsemigroup of $X^X$, the semigroup of all functions $X\rightarrow X$ under composition, and if $P$ is a subsemigroup of $\Omega^\Omega$, then $G\wr P$ consists of all functions $\Omega\times X \rightarrow \Omega\times X$ of the form
\[
\sigma(g_{\omega})_{\omega\in\Omega}: (x,\omega)\mapsto (x^{f_{\omega^\sigma}},\omega^\sigma).
\]
Before we proceed to explain how the cycle structure of a wreath product element may be determined, we clarify how the concept of a wreath product relates to cascaded connections of FSRs.

\begin{proposition}\label{prop 2.10}
Let $\FSR(f)$ be an $n$-stage FSR, and let $\FSR(g)$ be an $m$-stage FSR. Moreover, let $\Vec{t}=(0,\ldots,0,1)^T\in\IF_2^m$, and denote by $\rho(\Vec{t}):\IF_2^m\rightarrow\IF_2^m$ the translation $\Vec{v}\mapsto\Vec{v}+\Vec{t}$. Finally, denote by $\pi_1:\IF_2^n\rightarrow\IF_2$ the projection onto the first coordinate.

The transition function $\tilde{f}\ast\tilde{g}$ of the associated cascaded connection is equal to the wreath product element
\[
\tilde{f}\cdot(\tilde{g}\rho(\Vec{t})^{\pi_1(\Vec{y})})_{\Vec{y}\in \mathbb{F}_2^n} \in ({\IF_2^m}^{(\IF_2^m)}) \wr ({\IF_2^n}^{(\IF_2^n)}).
\]
Moreover, if both $\FSR(f)$ and $\FSR(g)$ are periodic, then so is $\FSR(f;g)$, and $\tilde{f}\ast\tilde{g}$ then lies in the permutational imprimitive wreath product $\Sym(\IF_2^m)\wr\Sym(\IF_2^n)$.
\end{proposition}

\begin{proof}
Let $X:=\IF_2^m$ and $\Omega:=\IF_2^n$. Then $\IF_2^{m+n}=X\times\Omega $, and note the following about $\tilde{f}\ast\tilde{g}$:
\begin{enumerate}
    \item It maps each copy
    \begin{align*}
        X \times \{(y_0, y_1,\ldots, y_{n-1})^{T}\} = \{(x_0, x_1,\ldots, x_{m-1}, y_0,\ldots, y_{n-1})^{T} : x_0,\ldots,x_{m-1} \in \mathbb{F}_2\}
    \end{align*}
    of $ X $ into another such copy, namely to $ X \times \{\tilde{f}(y_0, y_1,\ldots, y_{n-1})^{T}\} $. Moreover, if $\FSR(f)$ is periodic, then $\tilde{f}$ is surjective onto $\mathbb{F}_2^n$, whence each copy is \enquote{hit} in that case.
    \item After mapping $X\times\{\Vec{y}\} $ to $X\times\{\tilde{f}(\Vec{y})\}$ without altering the $X$-part, we can get the actual $(\tilde{f}\ast\tilde{g})$-value by applying either $\tilde{g}$ or $\tilde{g}\rho(\Vec{t})$ to the $X$-part, depending on whether or not $y_0=\pi_1(\Vec{y})=0$. Moreover, if $\FSR(g)$ is periodic, then both $\tilde{g}$ and $\rho(\Vec{t})$ are surjective onto $\mathbb{F}_2^m$, whence each point in each block is \enquote{hit}.   
\end{enumerate}
Because $\Sym(\IF_2^m)\wr\Sym(\IF_2^n)$ is the set intersection of $\Sym(\IF_2^{m+n})$ and $({\IF_2^m}^{(\IF_2^m)}) \wr ({\IF_2^n}^{(\IF_2^n)})$, the proof is complete.\qedhere
\end{proof}

Henceforth, for the sake of simplicity, we restrict ourselves to only considering \emph{periodic} FSRs (and cascaded connections thereof); these are also the most relevant for practical applications. In view of Lemma \ref{lemma 2.12}, if $\FSR(f)$ and $\FSR(g)$ are periodic FSRs, then understanding the periods of output sequences of the cascaded connection $\FSR(f;g)$ is essentially the same as understanding the cycle lengths of the transition function $\tilde{f}\ast\tilde{g}$. Formally, the cycle structure of a permutation is encoded in its so-called cycle type, defined as follows.

\begin{definition}
    The \textit{cycle type} of a permutation $ 
\sigma $ of a finite set $ X $ is the monomial
\begin{align*}
    \textnormal{CT}(\sigma) := x_1^{e_1}x_2^{e_2}\cdots x_{|X|}^{e_{|X|}} \in \mathbb{Q}[x_n : n \in \mathbb{N}^+] ,
\end{align*}
where $ e_{\ell} $ for $ \ell \in \{1, 2, \ldots, |X|\} $ is the number of cycles of length $ \ell $ of $ \sigma $.
\end{definition}

Specifically for understanding the cycle types of elements of imprimitive permutational wreath products, the following concept is crucial.

\begin{definition}\label{fcpDef}\cite[Definition 3.4]{Alexander}
    Let $G \leq \textnormal{Sym}(\Omega)$ be a permutation group on the finite set $\Omega$. Let $d$ be a positive integer, let $\psi \in \textnormal{Sym}(d)$, and let $g_0, g_1,\ldots, g_{d-1} \in G$. Consider the element $ g = (\psi, (g_0, g_1,\ldots, g_{d-1}))$ of the imprimitive permutational
 wreath product $G \wr \textnormal{Sym}(d)$. For each cycle $\zeta = (i_0, i_1,\ldots, i_{\ell-1})$ of $\psi$ we call an element of $G$ of the form
 \begin{align*}
    \fcp_{\zeta,i_0}(g) := g_{i_0}g_{i_1}\cdots g_{i_{\ell-1}}
 \end{align*}
 a \textit{forward cycle product} of $ g $ with respect to $ \zeta $.
\end{definition}

 In the context of Definition \ref{fcpDef}, we note that there may be several forward cycle products for a given cycle $\zeta$ and tuple $\vec{g}=(g_0,\ldots,g_{d-1})$, depending on the chosen starting point $i_0$ of $\zeta$. However, all forward cycle products for given $\zeta$ and $\vec{g}$ are cyclic shifts of each other and, in particular, they are conjugate elements of $G$ and thus have the same cycle type.
 
 The cycle types of elements of (imprimitive) permutational wreath products have been studied by P{\'o}lya~\cite{Polya}, and we describe the method in the following remark.

 \begin{remark}\label{remark 2.6}
  Let $ g = \sigma\cdot(f_{\omega})_{\omega\in \Omega} \in G \wr P $, and let $ \omega \in \Omega $. The set of all points on the $ \sigma $-cycle of $ \omega $ is denoted by $ \omega^{\langle \sigma \rangle} $. Then the subset $ M_{\omega} := \dot{\bigcup}_{\xi \in \omega^{\langle \sigma \rangle}}(X\times\{\xi\}) $ of $ X \times \Omega $ is mapped to itself by $ g $, so it is a disjoint union of cycles of $ g $. To determine the cycle type $\CT(g) $, proceed as follows:
 \begin{enumerate}
     \item Writing the cycle of $ \omega $ under $ \sigma $ as $ (\omega_0, \omega_1, \ldots, \omega_{\ell -1}) $, where $ \omega = \omega_0 $, compute the forward cycle product
     \begin{align*}
         f := f_{\omega_0}f_{\omega_1}\cdots f_{w_{\ell -1}} \in G \leq \Sym(X)
     \end{align*}
     and its cycle type $\CT(f) $.
     \item The cycle type of $ g|_{M_{\omega}} $ is the so-called $ \ell $-blow-up of $ \textnormal{CT}(f) $:
     \begin{align*}
         \CT(g|_{M_{\omega}}) = \BU_{\ell}(\CT(f)),
     \end{align*}
    where $ \BU_{\ell} $ is the unique $\IQ$-algebra endomorphism of $ \IQ[x_n : n\in \mathbb{N}^+] $ with $\BU_{\ell}(x_n) = x_{\ell\cdot n} $ for all $ n \in \mathbb{N}^+ $ (all cycle lengths are \enquote{blown up} by the factor $\ell$). 
    \item To get $\CT(g)$ as a whole, take the product of the cycle types computed in point (2) for all the (disjoint) cycles of $\sigma$ on $\Omega$.
 \end{enumerate}    
 \end{remark}

 \begin{example}
Let $ X = \{1,2,3,4,5\} $, $ \Omega = \{1,2,3\} $ and $ g = \sigma\cdot(f_1,f_2,f_3) $, where $ \sigma = (1,2,3) $, $ f_1 = (1,2,3)(4,5) $, $ f_2 = (1,2)(4,5) $, and $ f_3 = (1,4,3,2) $. Since $ \sigma $ is a long cycle, we have $ M_1 = X\times\Omega $ and $\CT(g) = \CT(g|_{M_1}) $. So,
\begin{align*}
    \CT(g) = \CT(g|_{M_1}) &= \BU_3(\CT(f_1\cdot f_2\cdot f_3))\\&= \BU_3(\CT((1,2,3)(4,5)\cdot(1,2)(4,5)\cdot(1,4,3,2))) \\&= \BU_3(\CT((1,4,3)))\\&= \BU_3(x_1^2x_3)\\&= x_3^2x_9.
\end{align*}
This means that $ g $ consists of two cycles of length 3 and one cycle of length 9.
 \end{example}

At this point, a comparison with the method of Mykkeltveit, Siu and Tong \cite[Section 2]{Mykkeltveit} seems appropriate. Please note that each cycle of the block permutation $\sigma$ is associated with a corresponding set of cycles of $g=\sigma(f_{\omega})_{\omega\in\Omega}$, which are just \enquote{blown up} cycles of an associated forward cycle product. A close look at \cite[Theorem 2.1]{Mykkeltveit} reveals that for a fixed cycle of $\sigma$, the associated set of \enquote{blown up} cycles is what Mykkeltveit, Siu and Tong call $C_j$ (when the cycle of $\sigma$ is represented by the periodic bit sequence $a_j$). The fact that the \enquote{blow-up factor} is the length of the associated cycle of $\sigma$ is reflected in their observation that the said cycle length divides the length of any associated blown up cycle, see \cite[Theorem 2.1(d)]{Mykkeltveit}. Finally, the sequences in Mykkeltveit-Siu-Tong's set $\theta(f)^{-1}(a_j)$ correspond to the cycles of a certain forward cycle product (fixed with the choice of the sequence $a_j$).

The upshot of this discussion is that the reduction argument furnished by \cite[Theorem 2.1]{Mykkeltveit} can, in fact, be traced back to much older ideas of P{\'o}lya, and as we will see in our Section \ref{results}, this conceptual algebraic perspective leads to new results on periods of cascaded connections of FSRs, the proofs of which are quite natural but would have been much harder to come up with while following the technical approach of \cite{Mykkeltveit}.

\section{Auxiliary algebraic results}\label{ch:Auxiliaries}

In this section, we recall a result of the first and third authors' paper~\cite{Bors} and then we state and prove some lemmas which are going to be used in the proofs of our main results in Section~\ref{results}.

 \begin{proposition}\cite[Proposition 2.1]{Bors}\label{prop 2.12}
     Let $q > 1$ be a power of a prime $p$, let $Q, U \in \mathbb{F}_q[X]$ with $ Q \neq X $
 monic irreducible, and let $e$ be a positive integer. Consider the affine permutation
$$ \lambda(X,U): R+(Q^e) \mapsto RX+U+(Q^e) $$
of $ \mathbb{F}_q[X]/(Q^e) $.
\begin{enumerate}
    \item If $ Q \neq X-1 $, then $\lambda(X,U)$ has the following cycle count (independently of $U$):
    \begin{itemize}
        \item 1 fixed point;
        \item $ \frac{q^{\deg(Q)}-1}{\textnormal{ord}(Q)} $ cycles of length $ \textnormal{ord}(Q) $;
        \item for each $ a = 1, 2 ,\ldots, \lceil \log_p(e) \rceil - 1 $: $ \frac{q^{p^{a-1}\deg(Q)}(q^{\deg(Q)p^{a-1}(p-1)}-1)}{p^a\textnormal{ord}(Q)} $ cycles of length $ \textnormal{ord}(Q)p^a $; and
        \item $\frac{q^{p^{\lceil \log_p(e)\rceil-1}\deg(Q)}(q^{\deg(Q)(e-p^{\lceil \log_p(e)\rceil-1})}-1)}{p^{\lceil \log_p(e)\rceil}\textnormal{ord}(Q)}$ cycles of length $ \textnormal{ord}(Q)p^{\lceil \log_p(e)\rceil} $.
    \end{itemize}
    \item If $ Q = X-1 $ and $ U + (Q^e) $ is a non-unit in $ \mathbb{F}_q[X]/(Q^e) $, then $ \lambda(X,U) $ has the following cycle count:
    \begin{itemize}
        \item $ q $ fixed points;
        \item for each $ a = 1, 2 ,\ldots, \lceil \log_p(e) \rceil - 1 $: $ \frac{q^{p^{a-1}}(q^{p^{a-1}(p-1)}-1)}{p^a} $ cycles of length $ p^a $; and
        \item $ \frac{q^{p^{\lceil \log_p(e)\rceil-1}}(q^{e-p^{\lceil \log_p(e)\rceil-1}}-1)}{p^{\lceil \log_p(e)\rceil}} $ cycles of length $ p^{\lceil \log_p(e)\rceil} $.
    
    \end{itemize}
    \item If $ Q = X-1 $ and $ U+(Q^e) $ is a unit in $ \mathbb{F}_q[X]/(Q^e) $, then $ \lambda(X,U) $ has $ \frac{q^e}{p^{\lfloor \log_p(e)\rfloor+1}} $ cycles, all of length $ p^{\lfloor \log_p(e)\rfloor+1} $.
\end{enumerate}
 \end{proposition}

We remark that the version of Proposition \ref{prop 2.12} published in \cite{Bors} contains a typographical error: in the last bullet point in statement (1), the first exponent of $q$ was wrongly typed as $p^{\lceil\log_p(e)\rceil}\deg(Q)$, as opposed to the correct value of $p^{\lceil\log_p(e)\rceil-1}\deg(Q)$ given here. Moreover, statement (3) in our Proposition \ref{prop 2.12} was unnecessarily split into two parts in \cite[Proposition 2.1]{Bors} using slightly different formulas ($\lceil\cdots\rceil$ instead of $\lfloor\cdots\rfloor$).

In what follows, we work with (primary) rational canonical forms of matrices, which we recall briefly. Let $K$ be a field. For a given monic degree $n$ univariate polynomial $P=X^n+a_{n-1}X^{n-1}+\cdots+a_1X+a_0\in K[X]$, the \emph{companion matrix of $P$}, written $\textnormal{Comp}(P)$, is the following $(n\times n)$-matrix over $K$:
\[
\textnormal{Comp}(P)=
\begin{pmatrix}
    0 & 0 & \cdots & 0 & -a_0 \\
    1 & 0 & \cdots & 0 & -a_1 \\
    0 & 1 & \cdots & 0 & -a_2 \\
    \vdots & \vdots & \vdots & \vdots & \vdots \\
    0 & 0 & \cdots & 1 & -a_{n-1}
\end{pmatrix}.
\]
Please note that $\textnormal{Comp}(P)$ is the matrix representing the ($K$-linear) modular multiplication by $X$ on the polynomial residue class ring $K[X]/(P)$ with respect to the $K$-basis $1+(P),X+(P),X^2+(P),\ldots,X^{n-1}+(P)$. One can turn the $K$-vector space $K^n$ into a $K[X]$-module by declaring that for all $Q\in K[X]$ and all $v\in K^n$, one has $Q\cdot v:=Q(A)v$. Using an important structural theorem which states that any module over a principal ideal domain (PID), such as $K[X]$, is a direct sum of cyclic modules over that PID, one obtains that the linear function $v\mapsto Av$ can be represented, with respect to a suitable $K$-basis of $K^n$, by a block diagonal matrix each block of which is the companion matrix of a suitable monic polynomial (of degree at most $n$) over $K$. There is a unique (up to reordering of the blocks) such matrix in which the number of diagonal blocks is minimal; this is called the \emph{rational canonical form of $A$} (for more details on the theory of rational canonical forms, see {\cite[Chapter 12]{Dummit}}).

On the other hand, if the monic polynomial $P\in K[X]$ admits the factorization $P=Q_1^{e_1}Q_2^{e_2}\cdots Q_r^{e_r}$ into powers of monic irreducible polynomials $Q_i\in K[X]$, then the Chinese Remainder Theorem for PIDs can be used to show that the matrix $\textnormal{Comp}(P)$ is similar to the block diagonal matrix with blocks $\textnormal{Comp}(Q_i^{e_i})$ for $i=1,2,\ldots,r$. Hence, any $(n\times n)$-matrix $A$ over $K$ also has a block diagonal form where each diagonal block is the companion matrix of a power of a monic irreducible polynomial over $K$. This latter block diagonal form is also unique up to reordering of its blocks and is called the \emph{primary rational canonical form of $A$}. Henceforth, we will work with primary rational canonical forms of matrices.

We wish to understand, for a fixed automorphism $\alpha$ of $V$, the different possibilities for the cycle types of affine permutations of $V$ of the form $\lambda(\alpha,v): x\mapsto x^{\alpha}+v$, where $v$ ranges over $V$. It turns out that these possibilities are in bijection with the possible $\alpha$-weights of the vectors $v$ in the sense of the following notation; see Proposition \ref{prop 2.15} below. 

\begin{notation}\label{notation 2.14}
    Let $q = p^f$ be a prime power, let $V$ be a finite-dimensional $\mathbb{F}_q$-vector space, and let $ \alpha $ be an $\mathbb{F}_q$-automorphism of $V$. Assume that the primary rational canonical blocks of $ \alpha $ are $ \textnormal{Comp}(Q_i^{e_i}) $ for $ i = 1,2,\ldots,s $, listed with multiplicities, and let $ V = \bigoplus_{i=1}^{s}V_i $ be a direct decomposition of $V$ into corresponding block subspaces.
For $i = 1,2,\ldots,s $, we define a subspace $W_i$ of $V_i$ as follows. If $Q_i\not=X-1$, then we simply set $W_i:=V_i$. On the other hand, if $Q_i=X-1$, then $W_i:=\im(\alpha_i-\id)$, the image of the linear function $\alpha_i-\id$, is the subspace consisting of the so-called \emph{$\alpha_i$-non-units} (an element of $V_i$ that does not lie in $W_i$ is called an \emph{$\alpha_i$-unit}). For $ v \in V $, denote by $v_i$ for $ i = 1,2,\ldots s $ the projection of $v$ to $V_i$. The $\alpha$-weight of $v$, denoted by $\textnormal{wt}_{\alpha}(v)$, is defined as follows:
\begin{align*}
    \textnormal{wt}_{\alpha}(v):= \textnormal{max}(\{0\}\cup\{1 + \lfloor \log_p(e_i) \rfloor :
1\leq i \leq s, v_i \not\in W_i\}).
\end{align*}    
\end{notation}

Of course, for those $i$ for which $Q_i\not=X-1$, it is impossible to have $v_i\notin W_i$, so only the indices $i$ for which $Q_i=X-1$ are relevant for the computation of the $\alpha$-weight of $v$. In particular, if there are no $i$ such that $Q_i=X-1$, then all $v\in V$ have $\alpha$-weight $0$.

Before we can prove Proposition \ref{prop 2.15}, we need some more preparations. Let us use the direct decomposition $V=\oplus_{i=1}^s{V_i}$ from Notation \ref{notation 2.14}, and denote by $\alpha_i$, respectively $v_i$, the restriction, respectively projection, of $\alpha$, respectively $v$, to $V_i$. Then we can view $\lambda(\alpha,v)$ as the component-wise application of the affine permutations $\lambda(\alpha_i,v_i)$ of the block subspaces $V_i$. As such, we have
\begin{align*}
        \textnormal{CT}(\lambda(\alpha,v)) =  \divideontimes_{i=1}^{s} \textnormal{CT}(\lambda(\alpha_i,v_i))
\end{align*}
where $\divideontimes$ is the $\mathbb{Q}$-bilinear product of polynomials introduced by Wei and Xu in \cite[Definition 2.2]{Wei-xu}. Specifically,
\[
\prod_i{x_i^{e_i}}\divideontimes\prod_j{x_j^{\epsilon_j}}
=
\prod_{i,j}{x_{\lcm(i,j)}^{e_i\epsilon_j\gcd(i,j)}}.
\]
For example,
\begin{align*}
(x_2^2x_3)\divideontimes(x_3^2x_4)
&=x_{\lcm(2,3)}^{2\cdot2\cdot\gcd(2,3)}\cdot x_{\lcm(2,4)}^{2\cdot 1\cdot\gcd(2,4)}\cdot x_{\lcm(3,3)}^{1\cdot 2\cdot\gcd(3,3)}\cdot x_{\lcm(3,4)}^{1\cdot 1\cdot\gcd(3,4)}
=x_6^4\cdot x_4^4\cdot x_3^6\cdot x_{12} \\
&=x_3^6x_4^4x_6^4x_{12}.
\end{align*}

The following result is sometimes useful when trying to simplify a Wei-Xu product, and we will use it in the proof of Proposition \ref{prop 2.15} below. Henceforth, we use the notation $\IN^+$ to denote the set of positive integers.

\begin{lemma}\label{absorptionLem}
Let
\[
\gamma=\prod_{\ell\in\mathbb{N}^+}{x_{\ell}^{k_{\ell}}}
\]
be the cycle type of a permutation on a finite number $m$ of points, and let
\[
\delta=\prod_{\ell\in\mathbb{N}^+}{x_{\ell}^{n_{\ell}}}
\]
be another cycle type of a finite permutation, not necessarily on $m$ points. Assume that
\begin{equation}\label{absorptionEq}
\operatorname{lcm}\{\ell\in\mathbb{N}^+: k_{\ell}\not=0\}\text{ divides }\operatorname{gcd}\{\ell\in\mathbb{N}^+: n_{\ell}\not=0\}.
\end{equation}
Then
\[
\gamma\divideontimes\delta=\delta^m.
\]
\end{lemma}

Formulated in words, condition (\ref{absorptionEq}) says that all cycle lengths of a permutation with cycle type $\gamma$ shall divide all cycle lengths of a permutation with cycle type $\delta$. If this happens, we say that \emph{$\delta$ absorbs $\gamma$}.

\begin{proof}[Proof of Lemma \ref{absorptionLem}]
Set
\[
\mathcal{L}_1:=\{\ell\in\mathbb{N}^+: k_{\ell}\not=0\}\text{ and }\mathcal{L}_2:=\{\ell\in\mathbb{N}^+: n_{\ell}\not=0\}.
\]
By the definition of $\divideontimes$ and assumption (\ref{absorptionEq}), we have
\begin{align*}
\gamma\divideontimes\delta &=(\prod_{\ell_1\in\mathcal{L}_1}{x_{\ell_1}^{k_{\ell_1}}})\divideontimes(\prod_{\ell_2\in\mathcal{L}_2}{x_{\ell_2}^{n_{\ell_2}}})=\prod_{\ell_1\in\mathcal{L}_1,\ell_2\in\mathcal{L}_2}{x_{\operatorname{lcm}(\ell_1,\ell_2)}^{\gcd(\ell_1,\ell_2)k_{\ell_1}n_{\ell_2}}}=\prod_{\ell_1\in\mathcal{L}_1,\ell_2\in\mathcal{L}_2}{x_{\ell_2}^{\ell_1k_{\ell_1}n_{\ell_2}}} \\
&=\prod_{\ell_2\in\mathcal{L}_2}{x_{\ell_2}^{mn_{\ell_2}}}=\delta^m,
\end{align*}
as required.
\end{proof}

A different way to see that Lemma \ref{absorptionLem} holds is by observing that if $\sigma\in\operatorname{Sym}(\Omega)$ is a permutation of cycle type $\gamma$ and $\psi\in\operatorname{Sym}(\Lambda)$ is a permutation of cycle type $\delta$, and if we consider the component-wise permutation $\sigma\times\psi\in\operatorname{Sym}(\Omega\times\Lambda)$, which is of cycle type $\textnormal{CT}(\sigma)\divideontimes\textnormal{CT}(\psi)$, then the cycle length of a point $(\omega,\lambda)\in\Omega\times\Lambda$ under $\sigma\times\psi$ is the least common multiple of the cycle lengths of $\omega$ under $\sigma$ and $\lambda$ under $\psi$ respectively, which equals the latter by condition (\ref{absorptionEq}). In other words, the first component $\omega\in\Omega$ has no influence on the cycle length of the pair $(\omega,\lambda)$.

At last, we are now ready to understand how weights relate to cycle types.

\begin{proposition}\label{prop 2.15}
Under the assumptions of Notation~\ref{notation 2.14}, the following hold: 
\begin{enumerate}
    \item If $v, w \in V$ have the same $ \alpha $-weight, then $\textnormal{CT}(\lambda(\alpha,v)) = \textnormal{CT}(\lambda(\alpha,w))$.
    \item For all $v \in V$ , the shortest cycle length of $\lambda(\alpha,v)$ that is a power of $p$ is $ p^{\textnormal{wt}_{\alpha}(v)} $. In particular, if $v, w \in V$ have different $ \alpha $-weights, then $\textnormal{CT}(\lambda(\alpha,v)) \neq \textnormal{CT}(\lambda(\alpha, w)) $.
\end{enumerate}
In particular, the number of distinct cycle types of affine permutations of $ V $ of the form $\lambda(\alpha,v)$ with $v$ ranging over $V$ is $|\textnormal{im}(\textnormal{wt}_{\alpha})|$, the number of distinct values of the function $\wt_{\alpha}:V\rightarrow\IN_0$, where $\IN_0$ denotes the set of nonnegative integers.
\end{proposition}

\begin{proof}
    Let $ \mathfrak{m} \in \{0,1,\ldots,\lfloor \log_p(\textnormal{dim}_{\mathbb{F}_q}(V))\rfloor \}$. As observed above, for a given vector $ v \in V $, we have
    \begin{align*}
        \textnormal{CT}(\lambda(\alpha,v)) =  \divideontimes_{i=1}^{s} \textnormal{CT}(\lambda(\alpha_i,v_i)).
    \end{align*}
    For the proof of statement (1), observe that by Proposition~\ref{prop 2.12}, the variation in the possible cycle types of $\lambda(\alpha,v)$ with $v \in V$ comes exclusively from the unipotent block subspaces of $\alpha$ in the chosen direct decomposition $ V = \bigoplus_{i=0}^{s}V_i $, i.e., from those block subspaces where $Q_i=X-1$. Moreover, by definition of the weight function $\textnormal{wt}_{\alpha}$, if we only consider vectors $v$ such that $ \textnormal{wt}_{\alpha}(v) = \mathfrak{m} $, then only the cycle types $\textnormal{CT}(\lambda(\alpha_i,v_i))$ stemming from blocks of the form $ \textnormal{Comp}((X-1)^{e_i}) $ with $ 1+\lfloor \log_p(e_i) \rfloor \leq \mathfrak{m} $ can be varied. If $ \mathfrak{m} = 0 $, then no variation is possible, and the statement is clear. If $ \mathfrak{m} > 0 $ then at least one of the blocks of $ \alpha $ of the form $\textnormal{Comp}((X-1)^{e_i})$ with $ 1+\lfloor \log_p(e_i) \rfloor = \mathfrak{m} $ yields the cycle type $ x_{p^{1+\lfloor \log_{p}(e_i)\rfloor}}^{q^{e_i}/p^{1+\lfloor \log_{p}(e_i)\rfloor}} $, and another look at Proposition~\ref{prop 2.12} shows that this cycle type absorbs all cycle types $\lambda(\alpha_j,v_j)$ coming from blocks $\textnormal{Comp}((X-1)^{e_j})$ with $ 1+\lfloor \log_p(e_j) \rfloor \leq \mathfrak{m} $, regardless of whether or not $v_j$ is an $\alpha_j$-unit. Therefore, despite possible variation at the individual block level, the overall cycle type of $\lambda(\alpha,v)$ is the same for all $v \in V$ with $ \textnormal{wt}_{\alpha}(v) = \mathfrak{m} $.

    In order to prove statement (2), we note that for each $v \in V$, by Proposition~\ref{prop 2.12}(1), all nontrivial cycle lengths in the cycle type $\lambda(\alpha_i,v_i)$ corresponding a non-unipotent block $\textnormal{Comp}(Q_i^{e_i})$ of $\alpha$ are divisible by a prime distinct from $p$. Since the cycle length of a point $u \in V$ under $\lambda(\alpha,v)$ is the least common multiple of the various cycle lengths of the projections $u_i \in V_i$ under $\lambda(\alpha_i,v_i)$, it follows that if the cycle length of $u$ under $V$ is a power of $p$, then the cycle length of $u_i$ under $\lambda(\alpha_i,v_i)$ must be 1 if $ Q_i \neq X-1 $. On the other hand, for unipotent blocks $\textnormal{Comp}((X-1)^{e_i})$, all cycle lengths of $\lambda(\alpha_i,v_i)$ are (possibly trivial, i.e., equal to $1$) powers of $p$ by Proposition~\ref{prop 2.12}(2,3), and we obtain the shortest cycle length of $\lambda(\alpha,v)$ that is a power of $p$ as the least common multiple of the shortest cycle lengths of the affine permutations $\lambda(\alpha_i,v_i)$ for those $i$ where $Q_i = X-1$. But the cycle types for those unipotent blocks where $v_i$ is an $\alpha_i$-non-unit have fixed points, whence the said least common multiple is 1 if there are no unipotent blocks where $v_i$ is an $\alpha_i$-unit. Otherwise, the said least common multiple is the largest cycle length that occurs in the cycle type of a unipotent block where $v_i$ is an $\alpha_i$-unit. In both cases, this least common multiple is represented by the expression $ p^{\textnormal{wt}_{\alpha}(v)} $, as required.
\end{proof}

We conclude this section with three lemmas, which together can be used to compute the rational canonical forms of powers of invertible matrices (equivalently, of iterates of vector space automorphisms). Of course, it suffices to do this for automorphisms that are given by the companion matrix of a single power of an irreducible polynomial; such automorphisms are called \emph{primary}. We will need the following notations.

\begin{notation}\label{exponentNot}
Let $q$ be a prime power, let $Q\in\mathbb{F}_q[X]$ be monic and irreducible, and let $\ell$ be an integer.
\begin{enumerate}
\item Let $\xi\in\mathbb{F}_{q^{\deg{Q}}}$ be one of the roots of $Q$. The minimal polynomial over $\mathbb{F}_q$ of the power $\xi^{\ell}$ is independent of the choice of $\xi$, and we denote this minimal polynomial by $\operatorname{pow}_{\ell}(Q)$.
\item We set $\operatorname{ind}_{\ell}(Q):=\frac{\operatorname{deg}{Q}}{\operatorname{deg}{\operatorname{pow}_{\ell}(Q)}}$.
\end{enumerate}
\end{notation}

Observe that for a given monic irreducible polynomial $Q\in\mathbb{F}_q[X]$ and integers $\ell_1,\ell_2$, one has
\[
\operatorname{pow}_{\ell_1}(\operatorname{pow}_{\ell_2}(Q))=\operatorname{pow}_{\ell_1\ell_2}(Q).
\]
If $\xi$ is a root of the monic irreducible polynomial $Q\in\mathbb{F}_q[X]$, then $\operatorname{ind}_{\ell}(Q)$ is equal to the field extension degree $[\mathbb{F}_q(\xi):\mathbb{F}_q(\xi^{\ell})]$.

\begin{lemma}\label{lem 2.16}
    Let $q = p^f$ be a prime power, and let $ \alpha $ be a primary $\mathbb{F}_q$-automorphism of a finite-dimensional $\mathbb{F}_q$ vector space $V$, with minimal polynomial $Q^e$ where $Q \in
\mathbb{F}_q[X]$ is monic and irreducible, and $ e \in \mathbb{N}^+ $. Moreover, let $ \ell $ be an integer that is coprime to $\textnormal{ord}(\alpha) = \textnormal{ord}(Q)\cdot p^{\lceil \log_p(e) \rceil}$. Then $ \alpha^{\ell} $ is also a primary $\mathbb{F}_q$-automorphism of $V$, and its minimal polynomial is $ \textnormal{pow}_{\ell}(Q)^e $.

\end{lemma}

\begin{proof}
Since $ \alpha $ and $ \alpha^{\ell} $ are powers of each other, they have the same invariant subspaces in $V$. Hence $ \alpha^{\ell} $ must be primary, for otherwise, $V$ would admit a nontrivial direct decomposition into $ \alpha $-invariant subspaces, contradicting that $ \alpha $ is primary. As for the minimal polynomial $ M $ of $ \alpha^{\ell} $ on $V$, note that by definition of $ \textnormal{pow}_{\ell}(Q) $, we have
\[
\textnormal{pow}_{\ell}(Q)(X^{\ell}) \equiv 0 \pmod{Q},
\]
which implies
\[
\textnormal{pow}_{\ell}(Q)^e(X^{\ell}) \equiv 0 \pmod{Q^e},
\]
so say $ \textnormal{pow}_{\ell}(Q)^e(X^{\ell}) = R\cdot Q^e $. Since $Q^e(\alpha) = 0_{\textnormal{End}(V)} $, it follows that 
\begin{align*}
    \textnormal{pow}_{\ell}(Q)^e(\alpha^{\ell}) = R(\alpha)Q^e(\alpha) = R(\alpha)0_{\textnormal{End}(V)} = 0_{\textnormal{End}(V)},
\end{align*}
whence $ M $ divides $ \textnormal{pow}_{\ell}(Q)^e $. But
\begin{align*}
    \textnormal{deg}~~M = \textnormal{dim}_{\mathbb{F}_q}(V) = e\cdot\textnormal{deg}~~Q = e\cdot\textnormal{deg}(\textnormal{pow}_{\ell}(Q)) = \textnormal{deg}(\textnormal{pow}_{\ell}(Q)^e),
\end{align*}
so that $  M = \textnormal{pow}_{\ell}(Q)^e $ as required.
\end{proof}

\begin{lemma}\label{lem 2.17}
Let $q = p^f$ be a prime power, and let $\alpha$ be a primary $\mathbb{F}_q$-automorphism of a finite-dimensional $\mathbb{F}_q$-vector space $V$, with minimal polynomial $Q^e$ where $Q \in
\mathbb{F}_q[X]$ is monic and irreducible, and $e \in \mathbb{N}^+$. Moreover, let $ \ell $ be a positive integer that
divides $\textnormal{ord}(\alpha)_{p'}$ (the product of all prime power factors in the factorization of $\textnormal{ord}(\alpha)$ except $p^{\nu_p(\textnormal{ord}(\alpha))}$). Then for each fixed $ \alpha $-unit $u \in V$, the $ \alpha^{\ell} $-cyclic subspaces $ U_i:= \mathbb{F}_qu^{\alpha^i\langle \alpha^{\ell}\rangle} $ of $ V $ for $ i = 0,1,\ldots,\textnormal{ind}_{\ell}(Q)-1 $ form an $ \alpha^{\ell} $-block subspace decomposition
of $V$, and $ (\alpha^{\ell})_{|U_i}$
has the minimal polynomial $ \textnormal{pow}_{\ell}(Q)^{e} $, independent of $ i $.
\end{lemma}

\begin{proof}
    Without loss of generality, we assume that $V = \mathbb{F}_q[X]/(Q^e)$, that $\alpha$ is the multiplication by $X$ modulo $Q^e$, and that $u = 1 + (Q^e)$. Set
    \begin{align*}
        \mathcal{E}:= \sum_{k=0}^{\frac{\textnormal{ord}(Q)}{\ell}-1}\mathbb{F}_{q}X^{k\ell},
    \end{align*}
    which is an $ \mathbb{F}_q $-subspace of $\mathbb{F}_q[X]$. If we view $\mathbb{F}_q[X]/(Q)$ as a copy of the finite field $ \mathbb{F}_{q^{\textnormal{deg}Q}} $, then the image $ \mathcal{E} \pmod{Q} $ of $ \mathcal{E} $ under the canonical projection $ \mathbb{F}_q[X] \rightarrow \mathbb{F}_q[X]/(Q) $ corresponds to the degree $ \textnormal{ind}_{\ell}(Q) $ subfield generated by $X^{\ell} + (Q)$. In particular, the reductions modulo $Q$ of the $\mathbb{F}_q$-subspaces $ X^{i}\mathcal{E} $ of $ \mathbb{F}_q[X] $ for $ i = 0,1,\ldots,\textnormal{ind}_{\ell}(Q)-1 $  form a direct decomposition of $\mathbb{F}_q[X]/(Q)$, because $\{X^i+(Q): i = 0,1,\ldots, \textnormal{ind}_{\ell}(Q)-1\}$ is an ($ \mathcal{E} \pmod{Q} $)-basis of $\mathbb{F}_q[X]/(Q)$.

    For $ i = 0,1,\ldots,\textnormal{ind}_{\ell}(Q)-1 $ denote by $U_i$ the $ \alpha^{\ell} $-cyclic subspace of $V = \mathbb{F}_q[X]/(Q^e)$ generated by $X^i+(Q^e)$. That is,
    \begin{align*}
        U_i = \sum_{m\in\mathbb{N}} \mathbb{F}_q(X^{i+\ell m}+(Q^e)).
    \end{align*}
We observe that $ \nu_{Q}(X^{\textnormal{ord}(Q)}-1) = 1 $. Indeed, if $ Q^2 \divides X^{\textnormal{ord}(Q)}-1 $, then $ Q^2\cdot R = X^{\textnormal{ord}(Q)}-1 $ for some $ R\in \mathbb{F}_q[X] $, so $ 
0 \neq \textnormal{ord}(Q)X^{\textnormal{ord}(Q)-1} = 2QQ'R+Q^2R' = Q(2Q'R+QR') $. This implies that $ Q \divides X^{\textnormal{ord}(Q)-1} $, hence $ Q = X $, a contradiction. Therefore, we can write $ X^{\textnormal{ord}(Q)} = 1+AQ $ for some $ A \in \mathbb{F}_q[X] $ not divisible by $ Q $. We claim that
\begin{align*}
    U_0 &= \mathbb{F}_q\{X^{k\ell}+(Q^e): k=0,1,\ldots,\frac{\textnormal{ord}(Q^e)}{\ell}-1\} \\&= \{\sum_{i=0}^{e-1}A^iE_iQ^i+(Q^e): E_0,E_1,\ldots,E_{e-1}\in \mathcal{E}\} =:U'_{0} \tag{1}\label{eq1}
\end{align*}
The first equality is by the definition of $U_0$, using that
$$ X^{\textnormal{ord}(Q^e)} \equiv 1 \pmod{Q^e}, $$ and thus
\begin{align*}
    \mathbb{F}_q(X^{k\ell+t~\textnormal{ord}(Q^e)}+(Q^e)) = \mathbb{F}_q(X^{k\ell}+(Q^e))
\end{align*}
for all $  k \in \{0,1,\ldots,\frac{\textnormal{ord}(Q^e)}{\ell}-1\}$ and all $ t \in \mathbb{N}_0 $.
To see that the second equality in formula (\ref{eq1}) holds, we consider a general element
\begin{align*}
    u_0:= \sum_{k=0}^{\frac{\textnormal{ord}(Q^e)}{\ell}-1} a_kX^{k\ell}+(Q^e)
\end{align*}
of $ U_0 $ (with $a_k \in \mathbb{F}_q$ for all $k$), and rewrite it as follows, using that $ 
\textnormal{ord}(Q^e) = \textnormal{ord}(Q)p^{\lceil \log_p(e)\rceil} $ by \cite[Chapter 3]{Lidl}: 
\begin{align*}
    u_0 &= \sum_{i=0}^{p^{\lceil \log_p(e)\rceil}-1} \sum_{j=0}^{\frac{\textnormal{ord}(Q)}{\ell}-1} a_{i\frac{\textnormal{ord}(Q)}{\ell}+j}X^{(i\frac{\textnormal{ord}(Q)}{\ell}+j)\ell}+(Q^e) \\&=  \sum_{i=0}^{p^{\lceil \log_p(e)\rceil}-1} X^{i\cdot\textnormal{ord}(Q)}\sum_{j=0}^{\frac{\textnormal{ord}(Q)}{\ell}-1} a_{i\frac{\textnormal{ord}(Q)}{\ell}+j}X^{j\ell}+(Q^e).
\end{align*}
For $ i = 0,1,\ldots,p^{\lceil \log_p(e)\rceil}-1 $, set
\begin{align*}
    E'_i:= \sum_{j=0}^{\frac{\textnormal{ord}(Q)}{\ell}-1}a_{i\frac{\textnormal{ord}(Q)}{\ell}+j}X^{j\ell}.
\end{align*}
    Observe that the $E'_i$ are arbitrary and independent elements of $ \mathcal{E} $. By the above, we have
    \begin{align*}
        u_0 &= \sum_{i=0}^{p^{\lceil\log_p(e)\rceil}-1}X^{i~\textnormal{ord}(Q)}E'_i+(Q^e) = \sum_{i=0}^{p^{\lceil\log_p(e)\rceil}-1} (1+AQ)^{i}E'_i + (Q^e) \\&= \sum_{i=0}^{p^{\lceil\log_p(e)\rceil}-1} \sum_{t=0}^{i} {i\choose t}A^tQ^tE'_i+(Q^e) = \sum_{t=0}^{p^{\lceil\log_p(e)\rceil}-1}A^tQ^t\sum_{i=t}^{p^{\lceil\log_p(e)\rceil}-1}{i\choose t}E'_i+(Q^e).
    \end{align*}
    Hence, setting
    \begin{align*}
        E_t:= \sum_{i=t}^{p^{\lceil\log_p(e)\rceil}-1}{i\choose t}E'_i \in \mathcal{E}, \tag{2} \label{eq2}
    \end{align*}
    we obtain the representation
    \begin{align*}
        u_0 = \sum_{t=0}^{p^{\lceil\log_p(e)\rceil}-1}A^tQ^tE_t+(Q^e) \in U'_0,
    \end{align*}
    which shows that $ U_0 \subseteq U'_0 $. Moreover, observe that by formula (\ref{eq2}), we have
    \begin{align*}
        & E_0 \in E'_0+\sum_{i>0}\mathbb{F}_qE'_i \\& E_1 \in E'_1+\sum_{i>1}\mathbb{F}_qE'_i,\\ &\vdots\\& E_{e-1} \in E'_{e-1}+\sum_{i>e-1}\mathbb{F}_qE'_i.
    \end{align*}
    This shows that in order to get a certain tuple of values in $ \mathcal{E} $ for $E_0, E_1,\ldots, E_{e-1}$, we can start by choosing $ E'_{e-1}, E'_e,\ldots,E'_{p^{\lceil \log_p(e)\rceil}-1} \in \mathcal{E} $ such that $ E_{e-1} $ assumes its desired value, and then successively adjust the values of $ E_{e-2}, E_{e-3},\ldots, E_1, E_0 $ via suitable choices of $ E'_{e-2}, E'_{e-3},\ldots,E'_1,E'_0 $, respectively. Hence $u_0$ can indeed assume arbitrary values in $ U'_0 $, and we conclude that $ U_0 = U'_0 $.

    Next, we claim that the minimal polynomial $M$ of $\alpha^{\ell}$ on $U_0$ is $ \textnormal{pow}_{\ell}(Q)^e $. To see that this holds, observe that as in the proof of Lemma~\ref{lem 2.16}, we have $ \textnormal{pow}_{\ell}(Q)^e(X^{\ell}) \equiv 0 \pmod{Q^e} $, which shows that $ \textnormal{pow}_{\ell}(Q)^e $ annihilates $ \alpha^{\ell} $ on all of $ \mathbb{F}_q[X]/(Q^e) $, in particular on $ U_0 $. This shows that $M$ divides $ \textnormal{pow}_{\ell}(Q)^e $, and thus $ M =  \textnormal{pow}_{\ell}(Q)^k $ for some $ k \in \{1,2,\ldots,e\} $.

    In order to show that $k = e$, we need to verify that $ \textnormal{pow}_{\ell}(Q)^{e-1} $ does not annihilate $ \alpha^{\ell} $ on~$ U_0 $. Assume otherwise. Since $1 + (Q^e) \in U_0$, it follows that 
    \begin{align*}
        (Q^e) = 0_{U_0} = \textnormal{pow}_{\ell}(Q)^{e-1}(X^{\ell})\cdot(1+(Q^e)) = \textnormal{pow}_{\ell}(Q)^{e-1}(X^{\ell})+(Q^e),
    \end{align*}
    which implies that $Q^e$ divides $ \textnormal{pow}_{\ell}(Q)^{e-1}(X^{\ell}) $. This is only possible if $ \nu_{Q}(\textnormal{pow}_{\ell}(Q)(X^{\ell})) > 1 $, i.e., if $ Q^2 $ divides $ \textnormal{pow}_{\ell}(Q)(X^{\ell}) $. From this, we conclude that
    \begin{align*}
        Q \divides \frac{d}{dX}\textnormal{pow}_{\ell}(Q)(X^{\ell}) = \textnormal{pow}_{\ell}(Q)'(X^{\ell})\cdot \ell X^{\ell-1}.
    \end{align*}
Since $ p \not\divides \ell $ and $ Q \neq X $, this implies that 
\begin{align*}
    Q \divides \gcd(\textnormal{pow}_{\ell}(Q)(X^{\ell}),\textnormal{pow}_{\ell}(Q)'(X^{\ell})).
\end{align*}
An equivalent way of stating this last formula is the following: If $ \xi \in \mathbb{F}_{q^{\textnormal{deg}Q}} $ is any of
the roots of $Q$, then $ \xi^{\ell} $  is a root of both $ \textnormal{pow}_{\ell}(Q) $ and $ \textnormal{pow}_{\ell}(Q)' $. That $ \xi^{\ell} $ is a root of the irreducible polynomial $ \textnormal{pow}_{\ell}(Q) $ shows that $ \xi^{\ell} $ is algebraic of degree $ \textnormal{deg~pow}_{\ell}(Q) $ over $ \mathbb{F}_q $, whence it is impossible for it to be also a root of the non-zero polynomial $ \textnormal{pow}_{\ell}(Q)' $ of strictly smaller degree, a contradiction. This completes our argument
that the minimal polynomial of $ \alpha^{\ell} $ on $U_0$ is $ \textnormal{pow}_{\ell}(Q)^{e} $.

Note that since $U_0$ is by definition an $ \alpha^{\ell} $-cyclic subspace of $V$, the $\mathbb{F}_q$-dimension of $U_0$ is equal to the degree of the minimal polynomial of $ \alpha^{\ell} $ on $U_0$. Hence
\begin{align*}
    \textnormal{dim}_{\mathbb{F}_q}(U_0) = \textnormal{deg~pow}_{\ell}(Q)^e = e\cdot \frac{\textnormal{deg}(Q)}{\textnormal{ind}_{\ell}(Q)} = \frac{\textnormal{dim}_{\mathbb{F}_q}V}{\textnormal{ind}_{\ell}(Q)}.
\end{align*}
Moreover, since each subspace $U_i$ for $ i = 0,1,\ldots,\textnormal{ind}_{\ell}(Q)-1 $ is an iterated image of $U_0$ under $ \alpha $, an $\mathbb{F}_q$-automorphism of $V$ that commutes with $ \alpha^{\ell} $, it follows that each $U_i$
is $\alpha^{\ell}$-cyclic and $\alpha^{\ell}$ has the minimal polynomial $ \textnormal{pow}_{\ell}(Q)^{e} $ on $ U_i $ (in particular, $ \textnormal{dim}_{\mathbb{F}_q}U_i = \textnormal{dim}_{\mathbb{F}_q}U_0 $).

It remains to show that $ V = \bigoplus_{i=0}^{\textnormal{ind}_{\ell}(Q)-1} U_i $. Because $\sum_{i=0}^{\textnormal{ind}_{\ell}(Q)-1}{\textnormal{dim}(U_i)}=\textnormal{dim}(V)$, it suffices to prove that $V=\sum_{i=0}^{\textnormal{ind}_{\ell}(Q)-1} U_i$. To see this, observe that by construction, we have
\begin{align*}
    \mathbb{F}_q[X]/(Q) = \sum_{i=0}^{\textnormal{ind}_{\ell}(Q)-1}(U_i\pmod{Q}). \tag{3} \label{eq3}
\end{align*}
In order to write an arbitrary element of $V$,
\begin{align*}
    v = \sum_{j=0}^{e-1}v_jQ^j+(Q^e),
\end{align*}
where $v_j \in \mathbb{F}_q[X]$ with $\textnormal{deg}~v_j < \textnormal{deg}~Q$, as a sum of elements
\begin{align*}
    u_i = X^i\sum_{j=0}^{e-1}A^jE_{i,j}Q^j+(Q^e) \in U_i
\end{align*}
for $ i = 0,1,\ldots,\textnormal{ind}_{\ell}(Q)-1 $ we take a recursive approach: we successively choose for $j = 0,1,\ldots, e-1$ the values of the “$j$-th digits” $E_{i,j}$ of $v_i$ modulo $ Q $ such that the “$j$-th digit” $ v_j = (A^j\sum_{i=0}^{\textnormal{ind}(Q)-1}X^iE_{i,j}) $ mod $ Q $ assumes the desired value (which is possible by formula (\ref{eq3})). Note that this may cause some constant carry-overs to later “digits” $v_{j'}$ with $j'>j$, which is not a problem as long as those carry-overs are accounted for when adjusting the values of those digits.
 \end{proof}

\begin{lemma}\label{lemma 2.18}
    Let $q = p^f$ be a prime power, and let $\alpha$ be a primary $\mathbb{F}_q$-automorphism
of a finite-dimensional $\mathbb{F}_q$-vector space $V$, with minimal polynomial $Q^e$ where $Q \in \mathbb{F}_q[X]$ is monic and irreducible, and $e \in \mathbb{N}^+$. Moreover, let $ \ell $ be a positive integer that
divides $\textnormal{ord}(\alpha)_p$ (in particular, $ \ell $ is a power of $p$). Set $\ell':=\min(\ell,e)$, and write $\textnormal{dim}_{\mathbb{F}_q}
(V) = e~\textnormal{deg}Q = a\ell'+b $
with $a, b \in \mathbb{Z}$, $a>0$ and $0 \leq b \leq \ell'-1$. The following hold:
\begin{enumerate}
    \item The primary rational canonical form of $\alpha^{\ell}$ on $V$ has $ \ell' $ blocks; $b$ of them are of the form $\textnormal{Comp}(\textnormal{pow}_{\ell}(Q)^{a+1})$, and the other $ \ell'-b $ blocks are of the form $\textnormal{Comp}(\textnormal{pow}_{\ell}(Q)^{a})$.

    \item If $Q = X-1$, then for each fixed $\alpha$-unit $u \in V$, the $ \mathbb{F}_q $-subspaces
    \begin{align*}
        V_i:= \bigoplus_{j=0}^{\lfloor \frac{e-1-i}{\ell}\rfloor} \mathbb{F}_qu^{(\alpha-1)^{i+\ell j}}~~~~for~~~ i=0,1,\ldots,\ell'-1
    \end{align*}
form an $ \alpha^{\ell} $-block subspace decomposition of $V$ . The minimal polynomial of $ (\alpha^{\ell})_{|V_i} $ is $ (X-1)^{a+1} $ if $ i \in \{0,1,\ldots,b-1\} $, and it is $ (X-1)^a $ if $ i \in \{b,b+1,\ldots,\ell'-1\} $.
\end{enumerate}
\end{lemma}

\begin{proof}
   We assume without loss of generality that $V = \mathbb{F}_q[X]/(Q^e)$, that $\alpha$ is the multiplication by $ X $ modulo $Q^e$, and (for statement (2)) that $u = 1 + (Q^e)$.

   We start with the proof of statement (2). Note that under the above assumptions, we have
   \begin{align*}
       V_i:= \bigoplus_{j=0}^{\lfloor \frac{e-1-i}{\ell}\rfloor}\mathbb{F}_q((X-1)^{i+\ell j}+((X-1)^e)).
   \end{align*}
Since $ \{(X-1)^{k}+((X-1)^e): k = 0,1,\ldots, e-1 \}$ is an $\mathbb{F}_q$-basis of $V$, it is clear that the subspace sums as which the $V_i$ are defined are indeed direct, and that $ V = \bigoplus_{i=0}^{\ell-1}V_i $. Moreover, if $\ell\geq e$ (i.e., if $\ell'=e$), then $\ell=p^{\lceil\log_p(e)\rceil}=\textnormal{ord}(\alpha)_p=\textnormal{ord}(\alpha)$ necessarily, so that $\alpha^{\ell}=\textnormal{id}$, and the rest of statement (2) is thus easily verified under that assumption. For the remainder of the proof of statement (2), we assume that $\ell<e$, which implies that $\ell'=\ell$.

Let $ i \in \{0,1,\ldots,b-1\} $. Then
\begin{align*}
    (X^{\ell}-1)^{a+1}V_i &= (X-1)^{(a+1)\ell}\bigoplus_{j=0}^{\lfloor \frac{e-1-i}{\ell}\rfloor}\mathbb{F}_q((X-1)^{i+\ell j}+((X-1)^e)) \\&= \bigoplus_{j=0}^{\lfloor \frac{e-1-i}{\ell}\rfloor}\mathbb{F}_q((X-1)^{i+(a+j+1)\ell}+((X-1)^e)) \\&= \bigoplus_{j=0}^{\lfloor \frac{e-1-i}{\ell}\rfloor} \mathbb{F}_q\{((X-1)^e)\} = \{((X-1)^e)\} =  \{0_{V}\},  
\end{align*}
 where the third equality uses that
$$ i+(a+j+1)\ell\geq(a+1)\ell = a\ell+\ell>a\ell+b = e~\textnormal{deg}(X-1) = e. $$
It follows that the minimal polynomial of $\alpha^{\ell}$ on $V_i$ divides $(X-1)^{a+1}$. On the other hand, 
$$ (X^{\ell}-1)^{a}V_i \supseteq(X^{\ell}-1)^{a}\mathbb{F}_q((X-1)^i+((X-1)^e
)) = \mathbb{F}_q((X-1)^{a\ell+i}+((X-1)^e
)) \neq \{0_V\} $$ because $a\ell+i\leq a\ell+b-1< e$. This shows that the minimal polynomial of $\alpha^{\ell}$ on $V_i$ is exactly $(X-1)^{a+1}$. By a similar argument, the minimal polynomial of $ \alpha^{\ell} $ on $V_i$ is $(X-1)^a$ if $i \in \{b, b + 1,\ldots, \ell-1\}$. In particular, for each $i$, the degree of the minimal polynomial of $\alpha^{\ell}$ on $V_i$ equals the $\mathbb{F}_q$-dimension of $V_i$, whence $V_i$ is $\alpha^{\ell}$-cyclic, and statement (2) is proved.

For statement (1), observe that since $ \textnormal{pow}_{\ell}(Q)(X^{\ell}) \equiv 0 \pmod{Q} $, we have $ \textnormal{pow}_{\ell}(Q)^e(X^{\ell}) \equiv 0 \pmod{Q^e} $, whence the minimal polynomial of $ \alpha^{\ell} $ on $V$ divides $ \textnormal{pow}_{\ell}(Q)^e $. Let $ V = \bigoplus_{i=0}^{s-1}V_i $ be an $ \alpha^{\ell} $-block subspace decomposition of $V$. The minimal polynomial of $ \alpha^{\ell} $ on $V_i$ divides $ \textnormal{pow}_{\ell}(Q)^e  $ and thus is of the form $ \textnormal{pow}_{\ell}(Q)^{e_i} $ for some $ e_i \in \{1,2,\ldots,e\} $. By comparing the $\mathbb{F}_q$-dimensions of $V$ and $ \bigoplus_{i=0}^{s-1}V_i $, we have $ \sum_{i=0}^{s-1}e_i=e $.

According to Lemma~\ref{lem 2.17}, when raising $ \alpha^{\ell} $ to the $\textnormal{ord}(Q)$-th power, each block
\\ $\textnormal{Comp}(\textnormal{pow}_{\ell}(Q)^{e_i})$ of $ \alpha^{\ell} $ splits into $\textnormal{deg}~Q$ blocks $\textnormal{Comp}((X-1)^{e_i})$. Hence, for each $ \epsilon \in \{1,2,\ldots,e\} $, the number of blocks of $ \alpha^{\ell~\textnormal{ord}(Q)} $ of the form $ \textnormal{Comp}((X-1)^{\epsilon}) $ is equal to 
\begin{align*}
    \textnormal{deg}~(Q)\cdot|\{i\in\{0,1,\ldots,s-1\}: e_i=\epsilon\}|.
\end{align*}
On the other hand, also by Lemma~\ref{lem 2.17}, when we raise $ \alpha $ to the $\textnormal{ord}(Q)$-th power, its
only block $\textnormal{Comp}(Q^e)$ on $V$ splits into $\textnormal{deg}~Q$ blocks $\textnormal{Comp}((X-1)^e)$. Moreover, by the already proved statement (2) of this lemma, if we subsequently raise $ \alpha^{\textnormal{ord}(Q)} $ to the $ \ell $-th power, each of these blocks $\textnormal{Comp}((X-1)^e)$ splits into $b$ blocks $\textnormal{Comp}((X-1)^{a+1})$
and $\ell'-b$ blocks $\textnormal{Comp}((X-1)^{a})$. It follows that $ \alpha^{\ell~\textnormal{ord}(Q)} $ has
\begin{itemize}
    \item $b~~\textnormal{deg}~Q$ blocks of the form $\textnormal{Comp}((X-1)^{a+1})$ and
    \item  $ (\ell'-b) $~$ \textnormal{deg}~Q $ blocks of the form $\textnormal{Comp}((X-1)^a)$.
\end{itemize}
Comparing with the above block count in terms of the $e_i$, we find that $s = \ell'$ and
  \[
   |\{i\in\{0,1,\ldots,\ell'-1\}: e_i=\epsilon\}| =
     \begin{cases}
        b  &\quad\text{if $ 
\epsilon = a+1 $,}\\
        \ell'-b  &\quad\text{if $ 
\epsilon = a  $,} \\ 0  &\quad\text{otherwise}. \\
     \end{cases}.
\]
In other words, the primary rational canonical form of $ \alpha^{\ell} $ has $b$ blocks of the form \\$\textnormal{Comp}(\textnormal{pow}_{\ell}(Q)^{a+1})$
and $ \ell'-b $ blocks of the form $\textnormal{Comp}(\textnormal{pow}_{\ell}(Q)^a)$, and no other blocks. This is just what we needed to show.
\end{proof}

\section{Cycle structures of cascaded FSRs via wreath products}\label{results}
In this section, we state our main results and prove them. In the following theorem, we give a nontrivial upper bound on the maximum period of a cascaded connection.

\begin{theorem} \label{thm 3.2}
Let $n$ and $m$ be positive integers, and let $\tilde{f}\in\Sym(\IF_2^n)$ and $\tilde{g}\in\Sym(\IF_2^m)$ be transition functions of FSRs (with characteristic functions $f$ and $g$, respectively). Unless $m=n=1$, the maximum period of an output sequence of the cascaded connection $\FSR(f;g)$ is at most
\begin{align*}
\textnormal{max}\{2^n(2^m-1), 2^m(2^n-1)\} = 2^{\textnormal{min}\{m,n\}}(2^{\textnormal{max}\{m,n\}}-1).
\end{align*}
\end{theorem}

\begin{proof}
We recall that by Lemma \ref{lemma 2.12}(2), periods of output sequences of $\FSR(f;g)$ are the same as cycle lengths of the transition function $\tilde{f}\ast\tilde{g}$, so we prove the asserted upper bound for those cycle lengths instead. Moreover, by Proposition~\ref{prop 2.10}, $\tilde{f}\ast\tilde{g}$ can be viewed as the wreath product element $\tilde{f}\cdot(\tilde{g}\rho(\Vec{t})^{\pi_1(\Vec{y})})_{\Vec{y}\in\IF_2^n} $, where $\Vec{t}=(0,\ldots,0,1)^T$. By Remark~\ref{remark 2.6}, we find that each cycle length of $ \tilde{f}\cdot(\tilde{g}\rho(\Vec{t})^{\pi_1(\vec{y})})_{\vec{y}\in\IF_2^n}$ is the product of a cycle length of $\tilde{f}$ with the cycle length of a certain permutation of $\IF_2^m$~(a certain forward cycle product). Therefore, if $\tilde{f}$ is not a $2^n$-cycle, then the maximum cycle length of $\tilde{f}\ast\tilde{g}$ can be at most $(2^n-1)\cdot2^m$, as desired. We may then assume that $\tilde{f}$ is a $2^n$-cycle (i.e., a De Bruijn cycle).~The associated forward cycle product is a product of $2^n$ permutations of $\IF_2^m$, of which $2^{n-1}$ are $\tilde{g}$ and $2^{n-1}$ are $\tilde{g}\rho(\Vec{t})$~(there are as many vectors in $\IF_2^n$ with the first bit 0 as there are with first bit 1). Denote by $\pi(\sigma)\in\IF_2$ the parity of $\sigma\in\Sym(\IF_2^m) $, i.e., $\pi(\sigma)=0$ if $\sigma$ is an even permutation (it can be written as a product of an even number of transpositions), and $\pi(\sigma)=1$ if $\sigma$ is odd. Then the parity of the said forward cycle product is
\[
2^{n-1}\pi(g)+2^{n-1}\pi(g\rho(\Vec{t}))=2^n\cdot\pi(g)+2^{n-1}\cdot\pi(\rho(\Vec{t})) \equiv 0 \pmod{2},
\]
using the assumption that $n>1$ or $m>1$~(note that $m>1$ causes $\rho(\Vec{t})$, which is a (disjoint) product of $2^{m-1}$ transpositions, to be an even permutation). The forward cycle product is thus an even permutation of $\IF_2^m$. In particular, it cannot be a $2^m$-cycle, so its maximum cycle length is at most $2^m-1$, and the maximum cycle length of $\tilde{f}\ast\tilde{g}$ is at most $2^n(2^m-1)$, as required.
\end{proof}

 Our second main result is a detailed and explicit description of the periods of output sequences of $\FSR(f;g)$ in the important special case where $\tilde{f}$ is a De Bruijn cycle and $\tilde{g}$ is a linear permutation. We focus on the cycle structure of the transition function $\tilde{f}\ast\tilde{g}$, in view of Lemma~\ref{lemma 2.12}. We also use, again, that by Proposition \ref{prop 2.10}, $\tilde{f}\ast\tilde{g}$ is the wreath product element
 \[
 \tilde{f}\cdot(\tilde{g}\rho(\Vec{t})^{\pi_1(\Vec{y})})_{\Vec{y}\in\IF_2^n} \in \Sym(\IF_2^m) \wr \Sym(\IF_2^n),
 \]
 where $\Vec{t}=(0,\ldots,0,1)^T$ and $\pi_1$ denotes the projection onto the first coordinate.
 
 It is not hard to check that the first $m$ iterates $\tilde{g}^k(\Vec{t})$ for $k=0,1,\ldots,m-1$ form an $\mathbb{F}_2$-basis of $\mathbb{F}_2^m$. Moreover, the representation matrix of $\tilde{g}$ with respect to this basis is $\Comp(P(X))$ for some monic polynomial $P(X)\in\IF_2[X]$ of degree $m$. Each composition $\tilde{g}\rho(\Vec{t})^{\pi_1(\Vec{y})}$ lies in the group $\AGL_m(2)$ of $\mathbb{F}_2$-affine permutations of $\mathbb{F}_2^m$, and so we can view $\tilde{f}\ast\tilde{g}$ more specifically as an element of $\AGL_m(2)\wr\Sym(\IF_2^n)$, which is more advantageous from a computational point of view. In order to describe the cycle type $\CT(\tilde{f}\ast\tilde{g})$, we note that by Remark~\ref{remark 2.6},
 \[
 \CT(\tilde{f}\ast\tilde{g}) = \CT(\tilde{f}\cdot(\tilde{g}\rho(\Vec{t})^{\pi_1(\Vec{y})})_{\Vec{y}\in\IF_2^n}) = \BU_{2^n}(\CT(\phi)),
 \]
 where $\BU_{2^n}$ is the $2^n$-blow-up function and $\phi$ is the following forward cycle product of $\tilde{f}\ast\tilde{g}$:
 \[
 \tilde{g}\rho(\Vec{t})^{\pi_1(\Vec{0})}\cdot \tilde{g}\rho(\Vec{t})^{\pi_1({\tilde{f}(\Vec{0})})}\cdots \tilde{g}\rho(\Vec{t})^{\pi_1({\tilde{f}^{2^n-1}(\Vec{0})})}.
 \]
Using the semidirect product structure of $\AGL_m(2) $, we can bring this product into the form 
\begin{align*}
\tilde{g}^{2^n}\cdot\rho(\sum_{k=0}^{2^n-1}\pi_1(\tilde{f}^k(\Vec{0})\Vec{t}^{(\tilde{g}^{2^n-1-k})})), \tag{I}\label{I}
\end{align*}
which is the standard form of this affine permutation (it is written as the product/composition of the linear permutation $\tilde{g}^{2^n}$ with a translation). We intend to understand the cycle type of (\ref{I}). Now, cycle types of affine permutations of finite vector spaces were described in ~\cite{Bors}, but we need to push the theory developed there further to understand the cycle type of the complex expression (\ref{I}) in terms of the involved parameters. First, we apply a certain isomorphism to bring everything into a nicer form. Specifically, we note once more that the representation matrix of $\tilde{g}\in\GL_m(2) $ with respect to the $ \mathbb{F}_2 $-basis $\tilde{g}^k(\Vec{t})$ for $k=0,1,\ldots,m-1$ is $\Comp(P(X))$ for some monic polynomial $P(X)\in \IF_2[X]$ of degree $m$. Under the $\IF_2$-vector space isomorphism $\IF_2^m \rightarrow \IF_2[X]/(P(X))$ with $\tilde{g}^k(\Vec{t}) \mapsto X^k+(P(X)) $ for $k=0,1,\ldots,m-1$, the linear automorphism $\tilde{g}$ corresponds to modular multiplication by $X$, and $\Vec{t}=(0,0,\ldots,0,1)^T$ corresponds to $1+(P(X)) $. Applying this isomorphism, the problem of determining the cycle type of (\ref{I}) turns into that of determining the cycle type of the affine permutation 
\[
\Gamma:R(X)+(P(X)) \mapsto X^{2^n}R(X)+\sum_{k=0}^{2^n-1}\pi_1(\tilde{f}^k(\Vec{0}))X^{2^n-1-k}\cdot 1+(P(X))
\]
of the ring $ \mathbb{F}_2[X]/(P(X)) $.
Now, if we factor $P(X)=Q_1(X)^{e_1}Q_2(X)^{e_2}\cdots Q_r(X)^{e_r}$ into pairwise coprime powers of irreducible polynomials $Q_i(X)$, then the Chinese Remainder Theorem tells us that there is a ring isomorphism
\[
\mathbb{F}_2[X]/(P(X)) \rightarrow \prod_{i=1}^r\IF_2[X]/(Q_i(X)^{e_i}),
\]
under which $\Gamma$ corresponds to the component-wise application of its modular reductions $\Gamma\bmod{Q_i(X)^{e_i}}$ for $i=1,\ldots,r$. Therefore, it suffices to study the cycle type of $\Gamma$ modulo each $ Q_i^{e_i} $, then take their $\divideontimes$-product as defined by Wei and Xu (see also the text passage before Lemma \ref{absorptionLem}). We fix $i$ and distinguish between two cases:
\begin{case}\label{case 1}
Assume $ Q_i(X) \neq X-1 $. By Lemmas \ref{lem 2.16} and \ref{lemma 2.18}, it can be proved that the automorphism $ R(X)+(Q_i(X)^{e_i}) \mapsto X^{2^n}R(X)+(Q_i(X)^{e_i}) $ of $\IF_2[X]/(Q_i(X)^{e_i}) $ has a primary rational canonical form with blocks of the form $\Comp(\pow_{2^n}(Q_i)^{e}) $, for some $ e \in \mathbb{N}^+ $ (not necessarily the same $ e $ for each block), where $ \pow_{2^n}(Q_i)$ is the minimal polynomial over $\IF_2$ of $\xi^{2^n}$, where $\xi$ is any of the roots of $Q_i$ in the algebraic closure $\overline{\IF_2}$. Because $\xi^{2^n}$ is the image of $\xi$ under the field automorphism $z\mapsto z^{2^n}$ of $\overline{\IF_2}$, it follows that
\[
\ord(\pow_{2^n}(Q_i))=\ord(\xi^{2^n})=\ord(\xi)=\ord(Q_i)>1.
\]
Therefore, we have $\pow_{2^n}(Q_i)\neq X-1 $, whence an application of Proposition~\ref{prop 2.12} yields that 
\[
\CT(\Gamma\bmod{Q_i^{e_i}}) = \CT(R(X)+(Q_i(X)^{e_i}) \mapsto X^{2^n}R(X)+(Q_i(X)^{e_i})),
\]
the cycle type of the $2^n$-fold iterate of the multiplication by $X$ modulo $Q_i(X)^{e_i}$. Now, the cycle type of the latter can be read off from Proposition~\ref{prop 2.12}, and cycle types of iterates can be computed via the following simple formula: if $\gamma=\CT(\sigma)=\prod_{\ell}{x_{\ell}^{e_{\ell}}}$, then for each $n\in\mathbb{Z}$, one has
\begin{equation}\label{itEq}
\ite_n(\gamma):=\CT(\sigma^n)=\prod_{\ell}{x_{\ell/\gcd(\ell,n)}^{e_{\ell}\cdot\gcd(\ell,n)}}.
\end{equation}
\end{case}

\begin{case}\label{case 2}
Assume $Q_i(X)=X-1 $. Set 
\[
\chi_{\tilde{f}}(X):= \sum_{k=0}^{2^n-1}\pi_1(\tilde{f}^k(\Vec{0}))X^{2^n-k-1}\in\mathbb{F}_2[X],
\]
a polynomial associated with the De Bruijn cycle $\tilde{f}$, chosen such that $\Gamma(R(X)+(P(X)))=X^{2^n}R(X)+\chi_{\tilde{f}}(X)+(P(X))$. We take note of the following two properties of $\chi_{\tilde{f}}(X):$
\begin{enumerate}
\item $\chi_{\tilde{f}}(X) $ is nonzero with $\deg{\chi_{\tilde{f}}}(X)<2^n $. More specifically, $\chi_{\tilde{f}}(X)$ has exactly $2^{n-1}$ terms, each of degree less than $2^n$, which correspond to the positions on the unique cycle of $\tilde{f}$ where a vector with first bit $1$ sits.
\item  $X-1$ divides $\chi_{\tilde{f}}(X)$ if and only if the number $2^{n-1}$ of terms of $\chi_{\tilde{f}}(X)$ is even if and only if $n>1$. 
\end{enumerate}
In order to describe $\CT(\Gamma\bmod{(X-1)^{e_i}}) $, we need to consider the following subcases:
\begin{subcase}\label{subcase 1}
Assume $2^n\geq\ord((X-1)^{e_i}) $. Then, since $\ord((X-1)^{e_i}) $ is a power of $2$, it follows that $\ord((X-1)^{e_i})\divides 2^n$. Hence, the automorphism $ R(X)+((X-1)^{e_i}) \mapsto X^{2^n}R(X)+((X-1)^{e_i}) $ of $ \mathbb{F}_2[X]/((X-1)^{e_i}) $ is the identity. It follows that $ \Gamma\bmod{(X-1)^{e_i}}$ is the additive shift $R(X)+((X-1)^{e_i}) \mapsto R(X)+\chi_{\tilde{f}}(X)+((X-1)^{e_i}) $, whence all of its cycles are of length
\[
\ord(\chi_{\tilde{f}}(X)+((X-1)^{e_i})) = 
\begin{cases}
1, &\text{if }(X-1)^{e_i} \mid \chi_{\tilde{f}}(X), \\
2, &\text{if }(X-1)^{e_i} \nmid \chi_{\tilde{f}}(X).
\end{cases}
\]
\end{subcase}

\begin{subcase}\label{subcase 2.2}
Assume $ 2^n < \ord((X-1)^{e_i}) $. Again, since $\ord((X-1)^{e_i}) $ is a power of $2$, this implies that $2^n \mid \ord((X-1)^{e_i}) $. Moreover, because $\ord((X-1)^{e_i})=2^{\lceil\log_2(e_i)\rceil}$ is the smallest power of $2$ that is at least $e_i$, the subcase assumption also implies that $2^n<e_i$. The block structure of the automorphism $R(X)+((X-1)^{e_i}) \mapsto X^{2^n}R(X)+((X-1)^{e_i})$ is described in Lemma~\ref{lemma 2.18}(2). Write $e_i=a\cdot 2^n+b $ with $ a,b \in \mathbb{N}_0, a>0, b\in\{0,1,\ldots,2^n-1\}$. The said automorphism has $ 2^n $ primary rational canonical blocks, and a corresponding block subspace decomposition is
\[
\IF_2[X]/((X-1)^{e_i}) = \bigoplus_{k=0}^{2^n-1}{V_k},
\]
where $V_k=\bigoplus_{j=0}^{\lfloor\frac{e_i-1-k}{2^n}\rfloor}\IF_2((X-1)^{k+j2^n}+(X-1)^{e_i})$ for $k=0,1,\ldots,2^n-1$. Moreover, the restriction of that automorphism to $V_k$ has the minimal polynomial
\[
\centering
\begin{cases}
(X-1)^{a+1},  &\text{if }k\in\{0,1,\ldots,b-1\}, \\
(X-1)^a,  &\text{if }k\in\{b,b+1,\ldots,2^n-1\},
\end{cases}
\]
of degree
\[
a_k=
\begin{cases}
a+1, & \text{if }k\in\{0,1,\ldots,b-1\}, \\
a, & \text{if }k\in\{b,b+1,\ldots,2^n-1\}.
\end{cases}
\]
This allows us to view, via the isomorphism $\mathbb{F}_2 [X]/((X-1)^{e_i}) \rightarrow\prod_{k=0}^{2^n-1}V_k$, the affine permutation $\Gamma\bmod{(X-1)^{e_i}}$ as a component-wise application of affine permutations $ 
A_k\in\Sym(V_k)$. We note that the linear part of $A_k$ is the restriction $\alpha_k$ of $R(X)+((X-1)^{e_i})\mapsto X^{2^n}R(X)+((X-1)^{e_i})$ to $V_k$, and the constant part is the corresponding projection of $\chi_{\tilde{f}}(X)+((X-1)^{e_i})$ to $ V_k$. For each fixed $k$, by Proposition~\ref{prop 2.12}, there are two possibilities for $\CT(A_k)$, depending on whether or not the said projection of $\chi_{\tilde{f}}(X)+((X-1)^{e_i})$ is a unit (i.e., does not lie in the image of $\alpha_k-\id$). Let us try to understand this more explicitly. Since $\alpha_k-\id$ is the restriction to $V_k$ of modular multiplication by $X^{2^n}-1=(X-1)^{2^n}$, we find, using the definition of $V_k$, that the $\IF_2$-subspace of non-units in
\[
V_k=\bigoplus_{j=0}^{\lfloor\frac{e_i-1-k}{2^n}\rfloor}\IF_2((X-1)^{k+j2^n}+(X-1)^{e_i})
\]
is
\[
W_k:=\bigoplus_{j=1}^{\lfloor\frac{e_i-1-k}{2^n}\rfloor}\IF_2((X-1)^{k+j2^n}+(X-1)^{e_i}).
\]
Now, by Proposition~\ref{prop 2.15}, the cycle type of $\Gamma\bmod{(X-1)^{e_i}}$ as a whole is bijectively determined by the maximum element of the set
\[
\{0\}\cup\{1+\lfloor \log_2(a_k)\rfloor: 0\leq k \leq 2^n-1, (\text{constant part of }A_k)\notin W_k\}
\]
To get a more concrete description, we distinguish between two subsubcases:

\begin{subsubcase}
Assume $a+1$ is not a power of $2$. Then $\{\lfloor\log_2(a_k)\rfloor: 0\leq k \leq 2^n-1\} = \{\lfloor\log_2(a)\rfloor\}$, so there are at most two distinct possibilities for $\CT(\Gamma\bmod{(X-1)^{e_i}})$, depending on whether or not there exists a $k\in\{0,1,\ldots,2^n-1\}$ such that the constant part of $A_k$ is a unit in $V_k$. Now, no such $k$ exists if and only if for each $k$, the projection of $\chi_{\tilde{f}}(X)+((X-1)^{e_i})$ to $V_k$ has vanishing $\IF_2((X-1)^k+(X-1)^{e_i})$-coordinate, i.e., if and only if $(X-1)^{2^n}$ divides $\chi_{\tilde{f}}(X)=\chi_{\tilde{f}}(X)\bmod{(X-1)^{e_i}}$.

However, that divisibility cannot hold, because $\chi_{\tilde{f}}(X)$ is a nonzero polynomial in $\IF_2[X]$ of degree at most $2^n-1$. We conclude that there is at least one primary rational canonical block of the linear part of $\Gamma\bmod{(X-1)^{e_i}}$ on which the cycle type of the corresponding restricted affine map associated with $\Gamma\bmod{(X-1)^{e_i}}$ is as in statement (3) of Proposition \ref{prop 2.12}, i.e., all cycles on that block are of length $2^{\lfloor\log_2(a)\rfloor+1}$. Moreover, following the proof of Proposition \ref{prop 2.15}, that cycle length is actually a multiple of all cycle lengths occurring in such a block, whence the said block absorbs all other blocks. Hence, in fact, \emph{all} cycles of $\Gamma\bmod{(X-1)^{e_i}}$ are of length $2^{\lfloor\log_2(a)\rfloor+1}$, i.e.,
\[
\CT(\Gamma\bmod{(X-1)^{e_i}})
=
x_{2^{\lfloor\log_2(a)\rfloor+1}}^{2^{e_i-\lfloor\log_2(a)\rfloor-1}}.
\]
\end{subsubcase}

\begin{subsubcase}
Assume $a+1$ is a power of $2$. Then
\[
\{\lfloor\log_2(a_k)\rfloor: 0\leq k\leq 2^n-1\} = \{\log_2(a+1)-1, \log_2(a+1)\},
\]
so a priori, there are three possibilities for $\CT(\Gamma\bmod{(X-1)^{e_i}})$, corresponding to the following cases:
\begin{enumerate}
\item There is no $k\in\{0,1,\ldots,2^n-1\} $ at all such that the constant part of $A_k$ is a unit of $V_k$, i.e., $(X-1)^{2^n} \mid \chi_{\tilde{f}}(X)$. As already argued in the previous subsubcase, this is actually impossible.
\item There is a $k\in\{b,b+1,\ldots,2^n-1\}$, but no $k\in\{0,1,\ldots,b-1\}$, such that the constant part of $A_k$ is a unit of $V_k$, i.e., $(X-1)^b \mid \chi_{\tilde{f}}(X) $ but $(X-1)^{2^n} \nmid \chi_{\tilde{f}}(X)$; of course, that indivisibility is always satisfied. Then one of the smaller, $a$-dimensional primary Frobenius blocks of the linear part of $\Gamma\bmod{(X-1)^{e_i}}$ corresponds to a cycle type of the form $x_{2^{\lfloor\log_2(a)\rfloor+1}}^{2^{a-\lfloor\log_2(a)\rfloor-1}}=x_{a+1}^{2^a/(a+1)}$ and absorbs all other $a$-dimensional blocks. Moreover, by Proposition \ref{prop 2.12}(2), the cycle lengths on each $(a+1)$-dimensional block are powers of $2$, and the largest among them is $2^{\lceil\log_2(a+1)\rceil}=a+1$. Hence, the said $a$-dimensional block also absorbs all $(a+1)$-dimensional blocks, and we have
\[
\CT(\Gamma\bmod{(X-1)^{e_i}})=x_{a+1}^{2^{e_i}/(a+1)}.
\]
\item There is a $k\in\{0,1,\ldots,b-1\} $ such that the constant part of $A_k$ is a unit of $V_k$, i.e., $(X-1)^b \nmid \chi_{\tilde{f}}(X)$. Then one of the larger, $(a+1)$-dimensional primary rational canonical blocks of the linear part of $\Gamma\bmod{(X-1)^{e_i}}$ absorbs all other blocks, whence
\[
\CT(\Gamma\bmod{(x-1)^{e_i}})
=
x_{2^{\lfloor\log_2(a+1)\rfloor+1}}^{2^{e_i-\lfloor\log_2(a+1)\rfloor-1}}
=
x_{2a+2}^{2^{e_i-1}/(a+1)}.
\]
\end{enumerate}
\end{subsubcase}
\end{subcase}    
\end{case}

The following theorem summarizes the results of the above discussion.

\begin{theorem}\label{longTheo}
Let $n$ and $m$ be positive integers, and let $\tilde{f}\in\Sym(\IF_2^n)$ and $\tilde{g}\in\Sym(\IF_2^m)$ be transition functions of FSRs such that $\tilde{f}$ is a De Bruijn cycle and $\tilde{g}$ is linear. Let $P(X)\in\IF_2[X]$ be the unique degree $m$ monic polynomial such that $\tilde{g}$ can be represented, with respect to a suitable $\IF_2$-basis of $\IF_2^m$, by the companion matrix $\Comp(P(X))$. Write $P(X)=(X-1)^{e_0}\prod_{i=1}^r{Q_i(X)^{e_i}}$ with $e_0\in\IN_0$, $e_i\in\IN^+$ for $1\leq i\leq r$, and $X-1,Q_1(X),\ldots,Q_r(X)\in\IF_2[X]$ being pairwise distinct monic irreducible polynomials. Recall/take note of the following notations.
\begin{itemize}
\item $\pi_1:\IF_2^n\rightarrow\IF_2$ for the projection to the first coordinate;
\item $\chi_{\tilde{f}}(X):=\sum_{t=0}^{2^n-1}{\pi_1(\tilde{f}^t(\vec{0})X^{2^n-1-t}}\in\IF_2[X]$;
\item $\divideontimes$ for the Wei-Xu product of polynomials in $\IQ[x_n: n\geq1]$, defined in \cite[Definition 2.2]{Wei-xu} (see also the text passage before Lemma \ref{absorptionLem} above);
\item $\BU_{\ell}$, where $\ell$ is a positive integer, for the unique $\IQ$-algebra endomorphism of $\IQ[x_n: n\geq1]$ such that $\BU_{\ell}(x_n)=x_{\ell\cdot n}$ for all positive integer $n$;
\item $\ite_t(\gamma)$ for the cycle type of the $t$-th iterate of any permutation with cycle type $\gamma$ (see also formula (\ref{itEq}) above);
\item $\Gamma\in\Sym(\IF_2[X]/(P(X)))$, $\Gamma(R(X)+(P(X)))=X^{2^n}R(X)+\chi_{\tilde{f}}(X)+(P(X))$;
\item $\Gamma_0$ and $\Gamma_+$ for the reductions of $\Gamma$ modulo $(X-1)^{e_0}$ and $\prod_{i=1}^r{Q_i(X)^{e_i}}$, respectively;
\item $\alpha_0$ for the multiplication by $X$ modulo $(X-1)^{e_0}$. The cycle type of $\alpha_0$ is described explicitly in Proposition \ref{prop 2.12}(2);
\item for $i=1,2,\ldots,r$: $\alpha_i$ for the multiplication by $X$ modulo $Q_i(X)^{e_i}$. The cycle type of $\alpha_i$ is described explicitly in Proposition \ref{prop 2.12}(1);
\item $\alpha_+$ for the multiplication by $X$ modulo $\prod_{i=1}^r{Q_i(X)^{e_i}}$, which satisfies $\CT(\alpha_+)=\divideontimes_{i=1}^r{\CT(\alpha_i)}$;
\item $e_0=a\cdot 2^n+b$ with $a,b\in\IZ$, $0\leq b<2^n$.
\end{itemize}
The following hold.
\begin{enumerate}
\item $\CT(\tilde{f}\ast\tilde{g})=\BU_{2^n}(\CT(\Gamma))$.
\item $\CT(\Gamma)=\CT(\Gamma_0)\divideontimes\CT(\Gamma_+)$.
\item $\CT(\Gamma_+)=\CT(\alpha_+^{2^n})=\ite_{2^n}(\CT(\alpha_+))$.
\item
\[
\CT(\Gamma_0)=
\begin{cases}
x_1, & \text{if }e_0=0; \\
x_1^{2^{e_0}}, & \text{if }e_0>0, n\geq\lceil\log_2(e_0)\rceil,\text{ and }(X-1)^{e_0}\mid\chi_{\tilde{f}}(X); \\
x_2^{2^{e_0-1}}, & \text{if }e_0>0, n\geq\lceil\log_2(e_0)\rceil,\text{ and }(X-1)^{e_0}\nmid\chi_{\tilde{f}}(X); \\
x_{2^{\lfloor\log_2(a)\rfloor+1}}^{2^{e_0-\lfloor\log_2(a)\rfloor-1}}, & \text{if }e_0>0, n<\lceil\log_2(e_0)\rceil, \text{ and }\log_2(a+1)\notin\IZ; \\
x_{a+1}^{2^{e_0}/(a+1)}, & \text{if }e_0>0, n<\lceil\log_2(e_0)\rceil, \log_2(a+1)\in\IZ, \text{ and }(X-1)^b\mid\chi_f(X); \\
x_{2a+2}^{2^{e_0-1}/(a+1)}, & \text{if }e_0>0, n<\lceil\log_2(e_0)\rceil, \log_2(a+1)\in\IZ,\text{ and }(X-1)^b\nmid\chi_f(X).
\end{cases}
\]
\end{enumerate}
\end{theorem}

In addition to the general method of describing the cycle structure of a cascaded connection of FSRs due to Mykkeltveit-Siu-Tong \cite[Theorme 2.1]{Mykkeltveit}, Chang-Gong-Wang \cite{Chang} focused on cascaded connections $\FSR(f;g)$ where $f$ is a De Bruijn cycle and $g$ is a linear permutation (the same situation as in Theorem \ref{longTheo}). Their general method, represented by \cite[Theorem 4]{Chang}, consists of reducing the determination of the cycle structure of the transition function of $\FSR(f;g)$ to the solution of a certain system of linear equations over $\IF_2$. While this provides an efficient method to determine the cycle structure for each given cascaded connection, it does not lead to explicit formulas as in our Theorem \ref{longTheo}.

Finally, we note the following theorem as a consequence of Theorem \ref{longTheo}. For polynomials $P(X),Q(X)\in\IF_2[X]$ where $Q(X)$ is irreducible, we denote by $\nu_{Q(X)}(P(X))$ the $Q(X)$-adic valuation of $P(X)$, i.e., the largest nonnegative integer $v$ such that $Q(X)^v$ divides $P(X)$ (and $\nu_{Q(X)}(0):=\infty$).

\begin{theorem}\label{shortTheo}
Let $n$ and $m$ be positive integers, and let $\tilde{f}\in\Sym(\IF_2^n)$ and $\tilde{g}\in\Sym(\IF_2^m)$ be transition functions of FSRs such that $\tilde{f}$ is a De Bruijn cycle and $\tilde{g}$ is linear. Let $P(X)\in\IF_2[X]$ be the unique degree $m$ monic polynomial such that $\tilde{g}$ can be represented, with respect to a suitable $\IF_2$-basis of $\IF_2^m$, by the companion matrix $\Comp(P(X))$.
\begin{enumerate}
\item If $\nu_{X-1}(P(X))\leq 1$ and $n>1$, then $\CT(\tilde{f}\ast\tilde{g}) = \BU_{2^n}(\CT(\tilde{g}^{2^n}))$.
\item If $\tilde{g}$ is of odd order and $n>1$, then $\CT(\tilde{f}\ast\tilde{g}) = \BU_{2^n}(\CT(g)) $.
\end{enumerate}
\end{theorem}

\begin{proof}
For statement (1): Let us use the same notations as in Theorem \ref{longTheo}. By that theorem, we have
\begin{equation}\label{ctEq}
\CT(\tilde{f}\ast\tilde{g})=\BU_{2^n}(\CT(\Gamma))=\BU_{2^n}(\CT(\Gamma_0)\divideontimes\CT(\Gamma_+))=\BU_{2^n}(\CT(\Gamma_0)\divideontimes\CT(\alpha_+^{2^n})).
\end{equation}
We argue that under our assumptions here, we have $\CT(\Gamma_0)=\CT(\alpha_0^{2^n})$. Indeed, note that $e_0\in\{0,1\}$. If $e_0=0$, then
\[
\CT(\Gamma_0)=x_1=\CT(\id_{\IF_2^0})=\CT(\id_{\IF_2^0}^{2^n})=\CT(\alpha_0^{2^n}).
\]
And if $e_0=1$, then we have $e_0>0$, $n\geq\lceil\log_2(e_0)\rceil$ and $(X-1)^{e_0}\mid\chi_f(X)$ (the latter by comment (2) after the definition of $\chi_f(X)$ above). Hence,
\[
\CT(\Gamma_0)=x_1^2=\CT(\id_{\IF_2})=\CT(\id_{\IF_2}^{2^n})=\CT(\alpha_0^{2^n}).
\]
This concludes the proof of the claim that $\CT(\Gamma_0)=\CT(\alpha_0^{2^n})$. Combining this with formula (\ref{ctEq}), we infer that
\begin{align*}
\CT(\tilde{f}\ast\tilde{g})&=\BU_{2^n}(\CT(\alpha_0^{2^n})\divideontimes\CT(\alpha_+^{2^n}))=\BU_{2^n}(\CT(\alpha_0^{2^n}\times\alpha_+^{2^n}))=\BU_{2^n}(\CT((\alpha_0\times\alpha_+)^{2^n})) \\
&=\BU_{2^n}(\CT(\tilde{g}^{2^n})).
\end{align*}
Here, the last equality uses the observation that because $\tilde{g}$ can be represented by $\Comp(P(X))$, its mapping behavior on $\IF_2^m$ corresponds, under a suitable $\IF_2$-vector space isomorphism, to that of the multiplication by $X$ modulo $P(X)$, which in turn corresponds (via the Chinese Remainder Theorem) to the component-wise modular multiplication by $X$ modulo $(X-1)^{e_0}$ (i.e., $\alpha_0$) and modulo $\prod_{i=1}^r{Q_i(X)^{e_i}}$ (i.e, $\alpha_+$), respectively.

Statement (2) follows as a special case from statement (1). Indeed, $\tilde{g}$ being of odd order (i.e., of order coprime to the characteristic $2$) is equivalent to its characteristic polynomial $P(X)$ being square-free. In particular, we have $\nu_{X-1}(P(X))\leq 1$, whence $\CT(\tilde{f}\ast\tilde{g})=\BU_{2^n}(\CT(\tilde{g}^{2^n}))$ by statement (1). Moreover, because $\tilde{g}$ is of odd order, we have
\[
\CT(\tilde{g}^{2^n})=\ite_{2^n}(\CT(\tilde{g}))=\CT(\tilde{g})
\]
by the formula for $\ite_t(\gamma)$ from above, which concludes the proof.
\end{proof}

We note that \cite[Corollary 2]{Chang} is a result that is closely related to, but essentially weaker than, our Theorem \ref{shortTheo}; please note that in contrast to that result, we do not require the assumption that $2^n\geq m$, only that $n>1$.

\section{Examples}\label{examples}
Throughout this section, we provide some examples of calculating the cycle type of $\tilde{f}\ast\tilde{g}$ to clarify our method.

\begin{example}\label{ex5.1}
Let $n=2$, $m=3$, and our feedback shift registers be given by the Boolean functions
\[
f_1(y_0,y_1)=y_0\oplus1\text{ and }g_1(x_0,x_1,x_2)=x_0\oplus x_1 .
\]
Now, let $a_{k+2}=a_k+1$, where $k\in\mathbb{N}^{+} $, be the recurrence relation corresponding to $f_1$, $b_{k+3}=b_k+b_{k+1}$, where $k\in\mathbb{N}^{+} $ be the recurrence relation corresponding to $g_1$, and $c_{k+3}=c_k+c_{k+1}+a_k$ the one corresponding to $\FSR(f;g)$.

Let the sequence generated by $\FSR(f)$ be $\underline{a}$, which is, up to cyclic shifts, $ \underline{a} = (0,0,1,1) $, a De Bruijn sequence. Now we calculate the sequences $\underline{c}$ generated by $\FSR(f;g)$:\\
    
    \begin{center}
    \begin{tabularx}{0.6\textwidth} { 
  | >{\raggedright\arraybackslash}X 
  | >{\centering\arraybackslash}X 
  | >{\raggedleft\arraybackslash}X | }
 \hline
 $\underline{a}$:~0011~0011~0011~0011~0011~0011~0011~0011~$\cdots$ \\
 \hline
 \underline{c}:~\textcolor{red}{0000}~0111~0101~1011~1110~0010~1001~\textcolor{red}{0000}~$\cdots$  \\
\hline
\underline{c}:~ 1100~1100~1100~$\cdots$\\
\hline
\end{tabularx}\newline 
\end{center}

As can be seen in the above table, $\FSR(f;g)$ produces exactly two distinct sequences up to cyclic shifts. The first one is of length $28$, and the second one is of length $4$. Hence, $\CT(\tilde{f}\ast\tilde{g})=x_4x_{28}$.

Now we use our method to find the cycle type of $\tilde{f}\ast\tilde{g}$. First, we note that $\tilde{f}:\IF_2^2 \rightarrow \IF_2^2 $ is defined by $ 
(y_1,y_2)^T \mapsto (y_2, f_1(y_0,y_1))^T $, and $ \tilde{g}: \IF_2^3 \rightarrow \mathbb{F}_2^3 $, is defined by $ (x_0,x_1,x_2)^T \mapsto (x_1,x_2,g_1(x_0,x_1,x_2))^T $. Also, we know, by Proposition~\ref{prop 2.10}, that $\tilde{f}\ast\tilde{g}$ can be viewed as the wreath product element $ \tilde{f}\cdot(\tilde{g}\rho(\Vec{t})^{\pi_1(\Vec{y})})_{\Vec{y}\in \IF_2^2} $. Let $\zeta$ be the following De Bruijn cycle corresponding to $\tilde{f}$.
\begin{center}    
\begin{tikzpicture}[node distance=2cm]
\node (X_1) {$(0,0)^T$};
\node (X_2) [right of=X_1] {$(0,1)^T$};
\node (X_3) [below of=X_2] {$(1,1)^T$};
\node (X_4) [left of=X_3] {$(1,0)^T$};
\node (end) [above of=X_4] {$(0,0)^T$};
\draw [->] (X_1) -- (X_2) ;
\node[text width=0.05cm,fill=white] at (0.95,0.28) {\tiny $f$};
\draw [->] (X_2) -- (X_3);
\node[text width=0.1cm,fill=white] at (2.2,-0.9) {\tiny $f$};
\draw [->] (X_3) -- (X_4);
\node[text width=0.1cm,fill=white] at (1,-1.7) {\tiny $f$};
\draw [->] (X_4) -- (end);
\node[text width=0.1cm,fill=white] at (-0.3,-0.9) {\tiny $f$};
\end{tikzpicture}
\end{center}
Now, by Remark~\ref{remark 2.6}, we have
\[
\CT(\tilde{f}\ast\tilde{g}) = \CT(\tilde{f}\cdot(\tilde{g}\rho(\Vec{t})^{\pi_1(\Vec{y})})_{\Vec{y}\in\IF_2^2}) = \BU_4(\CT(\fcp_{\zeta,(0,0)^T}(\tilde{f}(\tilde{g}\rho(\Vec{t})^{\pi_1(\Vec{y})})_{\Vec{y}\in \IF_2^2}))).
\]
Moreover,
\begin{align*}
\fcp_{\zeta,(0,0)^T}(\tilde{f}(\tilde{g}\rho(\Vec{t})^{\pi_1(\Vec{y})})_{\Vec{y}\in \IF_2^2}) &= \tilde{g}\rho(\Vec{t})^{\pi_1((0,0)^T)}\cdot \tilde{g}\rho(\Vec{t})^{\pi_1((0,1)^T)}\cdot \tilde{g}\rho(\Vec{t})^{\pi_1((1,1)^T)}\cdot \tilde{g}\rho(\Vec{t})^{\pi_1((1,0)^T)} \\&= \tilde{g}\cdot\tilde{g}\cdot \tilde{g}\rho(\Vec{t})\cdot\tilde{g}\rho(\Vec{t}) \\&= \tilde{g}^3\rho(\Vec{t})\tilde{g}\rho(\Vec{t}) = \tilde{g}^4\rho(\tilde{g}(\Vec{t})+\Vec{t}).
\end{align*}
Therefore, we intend to find $\CT(\tilde{g}^4\rho(\tilde{g}(\Vec{t})+\Vec{t}))=\CT(\Gamma)$, using the notation of Theorem \ref{longTheo}.

The representation matrix of $\tilde{g}$ with respect to the $\IF_2 $-basis $ \tilde{g}^k(\Vec{t}) $ for $ k = 0,1,2 $ is $\Comp(X^3+X+1)$. Note that since $ X^3+X+1 $ is irreducible, and $ X^3+X+1 \neq X+1 $, \textbf{Case~\ref{case 1}} of our discussion leading to Theorem \ref{longTheo} applies, i.e., $\Gamma=\Gamma_+$ in the notation of Theorem \ref{longTheo}. An application of that theorem thus yields that
\[
\CT(\Gamma)=\CT(\Gamma_+)=\ite_4(\CT(\alpha_+)),
\]
where $\alpha_+$ is the multiplication by $X$ modulo $X^3+X+1$ (the product of all irreducible factors of $P(X)=X^3+X+1$ that are distinct from $X-1$, taken with multiplicity). Because $X^3+X+1$ itself is irreducible, we may read $\CT(\alpha_+)$ off directly from Proposition \ref{prop 2.12}. Specifically, statement (1) of Proposition \ref{prop 2.12} with $q=2$, $e=1$ and using that $\ord(X^3+X+1)=7$ yields $\CT(\alpha_+)=x_1x_7^{(2^3-1)/7}=x_1x_7$. We conclude that
\[
\CT(\Gamma)=\ite_4(x_1x_7)=x_{1/\gcd(1,4)}^{1\cdot\gcd(1,4)}x_{7/\gcd(7,4)}^{1\cdot\gcd(7,4)}=x_1x_7
\]
and, consequently,
\[
\CT(\tilde{f}\ast\tilde{g})=\BU_4(x_1x_7)=x_4x_{28}.
\]
\end{example}

\begin{example}\label{ex.5.2}
Let $n=2$ and $f_1(y_0,y_1)=y_0\oplus 1$ as in Example~\ref{ex5.1}, but in this example, let $m=5$ and $g_1(x_0,x_1,x_2,x_3,x_4)=x_0\oplus x_1\oplus x_2$. Let $ a_{k+2}=a_k+1$ be the recurrence relation corresponding to $f_1$, and $b_{k+5}=b_k+b_{k+1}+b_{k+2}$ be the recurrence relation corresponding to $g_1$. Moreover, let $c_{k+5}=c_{k}+c_{k+1}+c_{k+2}+a_k$ be the recurrence relation corresponding to $\FSR(f;g)$, where $k\in\IN^+$. Like in the previous example, one can directly calculate the sequences generated by $\FSR(f;g)$. It generates the following four sequences, two of length 56 and two of length 8:
\begin{align*}
&\underline{c}_1 = [0, 0, 0, 0, 0, 0, 0, 1, 1, 0, 1, 1, 1, 0, 0, 0, 1, 1, 0, 0, 1, 0, 1, 0, 0, 0, 1, 0, 1, 1, 1, 1, \\&1, 1, 1, 0, 0, 1, 0, 0, 0, 1, 1, 1, 0, 0, 1, 1, 0, 1, 0, 1, 1, 1, 0, 1]\\&
\underline{c}_2 = [0, 0, 0, 0, 1, 0, 0, 0, 0, 1, 0, 1, 0, 1, 0, 0, 1, 1, 1, 0, 1, 1, 0, 1, 1, 0, 0, 1, 1, 1, 1, 1, \\&0, 1, 1, 1, 1, 0, 1, 0, 1, 0, 1, 1, 0, 0, 0, 1, 0, 0, 1, 0, 0, 1, 1, 0]\\&
\underline{c}_3 = [0, 0, 1, 1, 1, 1, 0, 0]\\&
\underline{c}_4 = [1, 0, 0, 1, 0, 1, 1, 0].
\end{align*}
Hence, $\CT(\tilde{f}\ast\tilde{g})=x_8^2x_{56}^2$.

Now we use our method to confirm the cycle type of $\tilde{f}\ast\tilde{g}$. Note that $\tilde{f}$ is the same as in Example~\ref{ex5.1}, and $\tilde{g}: \IF_2^5 \rightarrow \IF_2^5 $, is defined by
\[
(x_0,x_1,x_2,x_3,x_4)^T \mapsto (x_1,x_2,x_3,x_4,g_1(x_0,x_1,x_2,x_3,x_4))^T.
\]
Since $\tilde{f}$ is the same as in Example~\ref{ex5.1}, with the same argument we need to find $\CT(\tilde{g}^4\rho(\tilde{g}(\Vec{t})+\Vec{t}))=\CT(\Gamma)$; then $\CT(\tilde{f}\ast\tilde{g})$ is the $4$-blow-up of that.

Note that the representation matrix of $\tilde{g}$ with respect to the $\IF_2 $-basis $\tilde{g}^{k}(\Vec{t}) $ for $ k = 0,1,2,3,4 $ is $ \textnormal{Comp}(P(X))$ with $P(X)=X^5+X^2+X+1$. This time, $P(X)$ is not irreducible, but rather, it admits the factorization $P(X)=(X^3+X+1)(X+1)^2$. Because there are two distinct irreducible factors, we need to compute the cycle type of $\Gamma$ modulo each irreducible power, then take the Wei-Xu product of those two cycle types to obtain $\CT(\Gamma)$.

Now, the reduction $\Gamma\bmod{X^3+X+1}=\Gamma_+$ is the $\Gamma$ of the previous example, which lets us conclude without further calculations that $\CT(\Gamma_+)=x_1x_7$.

As for $\Gamma\bmod{(X+1)^2}=\Gamma_0$, we use statement (4) of Theorem \ref{longTheo} to compute it. Note that $e_0=2>0$, that $n=2\geq 1=\lceil\log_2(e_0)\rceil$ (which means that we are in Subcase 2.1 of the discussion leading to Theorem \ref{longTheo}) and that
\[
\chi_{\tilde{f}}(X)=\pi_1((0,0)^T)X^3+\pi_1((0,1)^T)X^2+\pi_1((1,1)^T)X+\pi_1((1,0)^T)=X+1
\]
is \emph{not} divisible by $(X+1)^{e_0}=(X+1)^2$. Therefore, the third case in statement (4) of Theorem \ref{longTheo} applies, letting us conclude that $\CT(\Gamma_0)=x_2^2$.

An application of statement (2) of Theorem \ref{longTheo} now yields that
\[
\CT(\Gamma)=\CT(\Gamma_0)\divideontimes\CT(\Gamma_+)=x_2^2\divideontimes(x_1x_7)=x_{\lcm(2,1)}^{2\cdot1\cdot\gcd(2,1)}x_{\lcm(2,7)}^{2\cdot1\cdot\gcd(2,7)}=x_2^2x_{14}^2.
\]
Finally, $\CT(\tilde{f}\ast\tilde{g})=\BU_4(\CT(\Gamma))=x_8^2x_{56}^2$.
\end{example}

\begin{example}
Let $ n = 2 $, $ f_1(y_0,y_1) = y_0\oplus 1 $ as in Example~\ref{ex5.1}, $ m = 8 $, and
\[
g_1(x_0,x_1,\ldots,x_7) = x_0 \oplus x_2\oplus x_3\oplus x_6\oplus x_7.
\]
By using a computer, we calculated the output sequences of $\FSR(f;g)$ directly. It generates 16 sequences of length $ 56 $, and 16 sequences of length $ 8 $. Hence, $\CT(\tilde{f}\ast\tilde{g}) = x_8^{16}x_{56}^{16}$.

Note that $\tilde{f}$ is the same as in Example~\ref{ex5.1}, and $\tilde{g}: \IF_2^8 \rightarrow \IF_2^8 $, is defined by
\[
(x_0,x_1,\ldots,x_7)^T \mapsto (x_1,\ldots,x_7,g_1(x_0,x_1,\ldots,x_7))^T.
\]
Since $\tilde{f}$ is the same as in Example~\ref{ex5.1}, with the same argument we need to find $\CT(\tilde{g}^4\rho(\tilde{g}(\Vec{t})+\Vec{t}))=\CT(\Gamma)$. 

Note that the representation matrix of $\tilde{g}$ with respect to the $\IF_2 $-basis $\tilde{g}^{k}(\Vec{t})$ for $k=0,1,\ldots,7$ is $\Comp(P(X))$ with $P(X)=X^8+X^7+X^6+X^3+X^2+1$. The factorization of $P(X)$ into powers of irreducible polynomials is $P(X)=(X^3+X+1)(X+1)^5$. Hence, $\Gamma_+=\Gamma\bmod{X^3+X+1}$, which is the $\Gamma_+$ of the previous example (and the $\Gamma$ of the first example), and we conclude immediately that $\CT(\Gamma_+)=x_1x_7$.

On the other hand, $\Gamma_0=\Gamma\bmod{(X+1)^5}$. We have $e_0=5$, and thus $n=2<3=\lceil\log_2(e_0)\rceil$, so we are in Subcase 2.2 of the discussion leading to Theorem \ref{longTheo}. Because $\tilde{f}$ has not changed from the previous example, we still have $\chi_{\tilde{f}}=X+1$. Integer division of $e_0=5$ by $2^n=4$ yields $5=1\cdot 4+1$, so $a=1$ and $b=1$. Because $\log_2(a+1)=\log_2(2)=1\in\IZ$ and $(X+1)^b=(X+1)^1\mid\chi_{\tilde{f}}(X)$, the penultimate case in statement (4) of Theorem \ref{longTheo} applies, letting us conclude that $\CT(\Gamma_0)=x_{a+1}^{2^{e_0}/(a+1)}=x_2^{16}$.

It follows that
\[
\CT(\Gamma)=\CT(\Gamma_0)\divideontimes\CT(\Gamma_+)=x_2^{16}\divideontimes(x_1x_7)=x_{\lcm(2,1)}^{16\cdot1\cdot\gcd(2,1)}x_{\lcm(2,7)}^{16\cdot1\cdot\gcd(2,7)}=x_2^{16}x_{14}^{16}
\]
and, finally,
\[
\CT(\tilde{f}\ast\tilde{g})=\BU_4(\CT(\Gamma))=x_8^{16}x_{56}^{16},
\]
confirming our direct calculations.
\end{example}

\section{Concluding remarks}\label{concRem}

We conclude this paper with two related open problems for further research.

In the paragraph after Theorem \ref{longTheo}, we compared our results on cycle types with Chang-Gong-Wang's method from \cite[Theorem 4]{Chang}, observing that their result does not lead to explicit formulas for the cycle types as our Theorem \ref{longTheo}. An advantage which Chang-Gong-Wang's method does have over ours at the moment is that it allows one to find, in each specific example, explicit representatives for the cycles of the transition function $\tilde{f}\ast\tilde{g}$, and thus explicit bit values with which the stages of the cascaded connection may be initialized to achieve a certain cycle length.

However, the wreath product method should allow one to achieve this as well, in the form of a general, parametric description of those cycle representatives. In fact, one should just be able to use an approach analogous to the one outlined in \cite[Section 3.1]{AQD}, which is for a different class of functions that can also be viewed as wreath product elements. This reduces the problem of finding cycle representatives for $\tilde{f}\ast\tilde{g}$ to that of finding such representatives for a certain affine permutation of $\IF_2^m$. We thus pose the following open problem, which we believe to be solvable using linear algebra:

\begin{problem}\label{openProb1}
Let $q$ be a prime power, $m$ a positive integer, $M$ an invertible $(m\times m)$-matrix over $\IF_q$, and $\vec{v}\in\IF_q^m$. In terms of $M$ and $\vec{v}$, give an explicit description of a set of representatives for the cycles of the affine permutation $A:\vec{x}\mapsto M\vec{x}+\vec{v}$ of $\IF_q^m$.
\end{problem}

We note two more things concerning Problem \ref{openProb1}. Firstly, for our application to FSRs, it would suffice to solve this for $q=2$. Secondly, in the desired explicit description, ideally each cycle representative is linked to the length of the associated cycle. That is, one should strive to find an explicit CRL-list for $A$ in the sense of \cite[Definition 1.2]{AQD}.

To motivate our second open problem, we note that according to Theorem \ref{longTheo}(4), the cycle type of $\Gamma_0$ often depends on whether a divisibility of the polynomial $\chi_{\tilde{f}}(X)$ by a certain power of $X-1$ holds. For that reason, understanding $\nu_{X-1}(\chi_{\tilde{f}}(X))$, the largest nonnegative integer $v$ such that $(X-1)^v$ divides $\chi_{\tilde{f}}(X)$, is of interest, and we pose the following open problem:

\begin{problem}\label{openProb2}
For each given positive integer $n$, define a set $M_n$ of nonnegative integers as follows:
\[
M_n:=\{\nu_{X-1}(\chi_{\tilde{f}}(X)): \tilde{f}\text{ is a De Bruijn cycle of }\IF_2^n\}.
\]
Describe the sets $M_n$. For example, what are the elements of $M_{100}$?
\end{problem}

\bibliographystyle{plain}
\makeatletter
\renewcommand*{\@biblabel}[1]{\hfill[#1]}
\makeatother
\bibliography{paper}
\end{document}